\documentclass[11pt,fleqn,mathserif]{article}
\usepackage{stmaryrd}
\usepackage{times}
\usepackage[top=2.5cm,bottom=2cm,left=2.5cm,right=2.5cm]{geometry}
\usepackage{amsmath,amssymb,mathrsfs,amsthm}

\usepackage[linkcolor=blue,pdfstartview=FitH,%ֱ������PDFLatexɾ��dvipdfm,
            CJKbookmarks=true,bookmarksnumbered=true,bookmarksopen=true,
            colorlinks=true,
            citecolor=blue]{hyperref}
\usepackage{graphicx,graphics,subfigure,float,caption2,booktabs}
\usepackage{varioref}
\newtheorem{theorem}{Theorem}
\newtheorem{lemma}{Lemma}
\newtheorem{remark}{Remark}
\newtheorem{example}{Example}
\newtheorem{property}{Property}

\newtheorem{definition}{Definition}

\numberwithin{equation}{section}
\begin{document}
\title{A Class of Second Order Difference Approximations for Solving Space Fractional Diffusion Equations}
\author{WenYi Tian,~~Han Zhou,~~Weihua Deng\footnote{Corresponding Author. E-mail: dengwh@lzu.edu.cn}\\[10pt]
        {School of Mathematics and Statistics, Lanzhou University,
         Lanzhou 730000, P. R. China}
       }
\date{}
\maketitle
%-----------------------------------------------------------------------------------------
\begin{abstract}
A class of second order approximations, called the weighted and shifted Gr\"{u}nwald
difference operators, are proposed for Riemann-Liouville fractional derivatives, with
their effective applications to numerically solving space fractional diffusion equations
in one and two dimensions. The stability and convergence of our difference schemes for
space fractional diffusion equations with constant coefficients in one and two dimensions
are theoretically established. Several numerical examples are implemented to testify the
efficiency of the numerical schemes and confirm the convergence order, and the numerical
results for variable coefficients problem are also presented.

\vskip 4pt \textbf{Keywords}: Riemann-Liouville fractional derivative, Fractional
diffusion equation, Weighted and shifted Gr\"{u}nwald difference operator.

\vskip 4pt \textbf{AMS subject classifications}: 26A33, 65L12, 65L20
\end{abstract}
%=========================================================================================
\section{Introduction}
Fractional calculus is a fundamentally mathematical tool for describing some special
phenomenons arising from engineering and science
\cite{Miller:93,Podlubny:99,Schneider:89}. One of its most important applications is to
describe the subdiffusion and superdiffusion process \cite{Chechkin:02,
Krepysheva:06,Negrete:03}. The suitable mathematical models are the diffusion equations
with time and/or space fractional derivatives, where the classical first order derivative
in time is replaced by the Caputo fractional derivative of order $\alpha\in (0, 1)$, and
the second order derivative in space is essentially replaced  by the Riemann-Liouville
fractional derivative of order $\alpha\in (1, 2]$. The physical interpretation and
practical applications of fractional diffusion equations have been discussed a lot with
some common ideas \cite{Barkai:02,Gorenflo:98,Metzler:00}. Based on these, our main
purpose of this paper is to study the higher accurate numerical solution of the space
fractional diffusion equation by a novel finite difference approximation.

From the perspective of the numerical analysis, there are some fundamental difficulties
in numerically approximating the fractional derivatives, because some good properties of
classical approximating operators are lost. Over the last decades, the finite difference
method has some developments in solving the fractional partial differential equations,
e.g., \cite{Celik:11,Meerschaert:04,Meerschaert:06, Yang:10}. The Riemann-Liouville
fractional derivative can be  discretized by the standard Gr\"{u}nwald-Letnikov formula
\cite{Podlubny:99} with only the first order accuracy, but the difference scheme based on
the Gr\"{u}nwald-Letnikov formula for time dependent problems is unstable
\cite{Meerschaert:04}. To overcome this problem, Meerschaert and Tadjeran in
\cite{Meerschaert:04} firstly proposed the shifted Gr\"{u}nwald-Letnikov formula to
approximate fractional advection-dispersion flow equations. Recently, second order
approximations to fractional derivatives are studied, Sousaa and Li presented a second
order discretization for Riemann-Liouville fractional derivative and established an
unconditionally stable weighted average finite difference method for one-dimensional
fractional diffusion equation in \cite{Sousaa:11}, and the results in two-dimensional
two-sided space fractional convection diffusion equation in finite domain can be seen in
\cite{Chen:12}; Ortigueira \cite{Ortigueira:06} gave the ``fractional centred derivative"
to approximate the Riesz fractional derivative with second order accuracy, and this
method was used by \c{C}elik and Duman in \cite{Celik:11} to approximate fractional
diffusion equation with the Riesz fractional derivative in a finite domain. In this
paper, we propose a more general and flexible approach to approximate the
Riemann-Liouville fractional derivative via combining the distinct shifted
Gr\"{u}nwald-Letnikov formulae with their corresponding weights, and the weighted and
shifted Gr\"{u}nwald-Letnikov formulae achieve second and higher order accuracy. A
detailed algorithm shows that the weights are related to not only the shifted numbers but
also the order of the fractional derivative, which implies the numerical algorithm is
more related to the equation itself.

The paper is briefly summarized as follows. In Sec. 2, we propose a class of discrete
operators to approximate the Riemann-Liouville fractional derivatives with high order
truncating errors. In Sec. 3 and 4, one dimensional and two dimensional fractional
diffusion equations are numerically solved by using the finite difference method based on
the weighted and shifted Gr\"{u}nwald-Letnikov formulae, and the stability analysis of
each case is presented. We prove that the finite difference solutions approximate the
exact ones with $O(\tau^2+h^2)$ in the discrete $L^2$ norm. Some numerical experiments
are performed in Sec. 5 to verify the efficiency and accuracy of the methods. And the
concluding remarks are given in the last Section.
%=========================================================================================
\section{High Order Approximations for Riemann-Liouville Fractional Derivatives}
We begin with the definitions of the Riemann-Liouville fractional derivatives and the properties of their Fourier transform.
\begin{definition}[\cite{Podlubny:99}]
The $\alpha\,(n-1<\alpha<n)$ order left and right Riemann-Liouville
fractional derivatives of the function $u(x)$ on $[a,b]$ are defined
as
\begin{itemize}
  \item[(1)] left Riemann-Liouville fractional derivative:
    \begin{equation*}
      _aD_x^{\alpha}u(x)=\frac{1}{\Gamma(n-\alpha)}\frac{\mathrm{d}^n}{\mathrm{d}
          x^n}\int_a^x\frac{u(\xi)}{(x-\xi)^{\alpha-n+1}}\mathrm{d}\xi;
    \end{equation*}
  \item[(2)] right Riemann-Liouville fractional derivative:
    \begin{equation*}
      _xD_b^{\alpha}u(x)=\frac{(-1)^n}{\Gamma(n-\alpha)}\frac{\mathrm{d}^n}{\mathrm{d}
          x^n}\int_x^b\frac{u(\xi)}{(\xi-x)^{\alpha-n+1}}\mathrm{d}\xi.
    \end{equation*}
\end{itemize}
 If $\alpha=n$, then $_aD_x^{\alpha}u(x)=\frac{\mathrm{d}^n}{\mathrm{d}x^n}u(x)$ and $_xD_b^{\alpha}u(x)=(-1)^n\frac{\mathrm{d}^n}{\mathrm{d}x^n}u(x)$.
\end{definition}
\begin{property}[\cite{Ervin:06}]
  Let $\alpha>0,~u\in C_0^{\infty}(\Omega),~\Omega\subset\mathbb{R}$. The Fourier transforms of the left and right Riemann-Liouville fractional derivatives satisfy
  \begin{align*}
    & \mathscr{F}({_{-\infty}}D_x^{\alpha}u(x))=(i\omega)^{\alpha}\hat{u}(\omega), \\
    & \mathscr{F}({_x}D_{\infty}^{\alpha}u(x))=(-i\omega)^{\alpha}\hat{u}(\omega),
  \end{align*}
  where $\hat{u}(\omega)$ denotes the Fourier transform of $u$,
  \begin{equation*}
    \hat{u}(\omega)=\int_{\mathbb{R}}\mathrm{e}^{-i\omega x}u(x)\mathrm{d}x.
  \end{equation*}
\end{property}
In \cite{Meerschaert:04}, the shifted Gr\"{u}nwald difference operator
\begin{equation}\label{eq:2.1}
  A_{h,p}^{\alpha}u(x)=\frac{1}{h^\alpha}\sum_{k=0}^{\infty}g_k^{(\alpha)}u(x-(k-p)h),
\end{equation}
approximates the Riemann-Liouville fractional derivative uniformly
with first order accuracy, i.e.,
\begin{equation}
  A_{h,p}^{\alpha}u(x)={_{-\infty}}D_x^{\alpha}u(x)+O(h),
\end{equation}
where $p$ is an integer and
$g_k^{(\alpha)}=(-1)^k\binom{\alpha}{k}$. In fact, the coefficients
$g_k^{(\alpha)}$ in \eqref{eq:2.1} are the coefficients of the power
series of the function $(1-z)^\alpha$,
\begin{equation}\label{eq:2.2}
  (1-z)^\alpha=\sum_{k=0}^\infty(-1)^k\binom{\alpha}{k}z^k=\sum_{k=0}^\infty g_k^{(\alpha)}z^k,
\end{equation}
for all $|z|\le1$, and they can be evaluated recursively
\begin{equation}
  g_0^{(\alpha)}=1,\quad g_k^{(\alpha)}=\Big(1-\frac{\alpha+1}{k}\Big)g_{k-1}^{(\alpha)}, ~ k=1,2,\ldots.
\end{equation}
\begin{lemma}[\cite{Meerschaert:04,Meerschaert:06,Podlubny:99}]\label{lem:1}
  The coefficients in \eqref{eq:2.1} satisfy the following properties for
  $1<\alpha\le2$,
  \begin{equation}\left\{
    \begin{split}
      & g_0^{(\alpha)}=1,~ g_1^{(\alpha)}=-\alpha<0, \\
      & 1\ge g_2^{(\alpha)}\ge g_3^{(\alpha)}\ge\ldots\ge0,\\
      & \sum_{k=0}^\infty g_k^{(\alpha)}=0,~\sum_{k=0}^m g_k^{(\alpha)}<0,~m\ge1.
    \end{split}\right.
  \end{equation}
\end{lemma}
%-----------------------------------------------------------------------------------------
\subsection{Second Order Approximations}
Inspired by the shifted Gr\"{u}nwald difference operator \eqref{eq:2.1} and multi-step
method, we derive the following second order approximation for the Riemann-Liouville
fractional derivatives.
\begin{theorem}\label{thm:1}
  Let $u\in L^1(\mathbb{R})$, ${_{-\infty}}D_x^{\alpha+2}u$ and its Fourier transform belong to $L^1(\mathbb{R})$, and define the weighted and shifted Gr\"{u}nwald difference (WSGD) operator by
  \begin{equation}\label{eq:2.3}
    {_L}\mathcal{D}_{h,p,q}^{\alpha}u(x)=\frac{\alpha-2q}{2(p-q)}A_{h,p}^{\alpha}u(x)
    +\frac{2p-\alpha}{2(p-q)}A_{h,q}^{\alpha}u(x),
  \end{equation}
  then we have
  \begin{equation}
    {_L}\mathcal{D}_{h,p,q}^{\alpha}u(x)={_{-\infty}}D_x^{\alpha}u(x)+O(h^2)
  \end{equation}
  uniformly for $x\in\mathbb{R}$, where $p,q$ are integers and $p\neq q$.
\end{theorem}
{\em Note.}  The role of $p$ and $q$ is symmetric, i.e.,
${_L}\mathcal{D}_{h,p,q}^{\alpha}u(x)={_L}\mathcal{D}_{h,q,p}^{\alpha}u(x).$
\begin{proof} [Proof of Theorem 1]
  By the definition of $A_{h,p}^{\alpha}$ in \eqref{eq:2.1}, we can rewrite the WSGD operator as
  \begin{equation}\label{eq:2.4}
     {_L}\mathcal{D}_{h,p,q}^{\alpha}u(x)=\frac{\alpha-2q}{2(p-q)}\frac{1}{h^\alpha}
     \sum_{k=0}^{\infty}g_k^{(\alpha)}u(x-(k-p)h)+
     \frac{2p-\alpha}{2(p-q)}\frac{1}{h^\alpha}\sum_{k=0}^{\infty}g_k^{(\alpha)}u(x-(k-q)h).
  \end{equation}
  Taking Fourier transform on \eqref{eq:2.4}, we obtain
  \begin{equation}
    \begin{split}\label{eq:2.12}
      \mathscr{F}[{_L}\mathcal{D}_{h,p,q}^{\alpha}u](\omega) & =\frac{1}{h^\alpha}
      \sum_{k=0}^{\infty}g_k^{(\alpha)}\Big(\frac{\alpha-2q}{2(p-q)}\mathrm{e}^{-i\omega(k-p)h}+
      \frac{2p-\alpha}{2(p-q)}\mathrm{e}^{-i\omega(k-q)h}\Big)\hat{u}(\omega) \\
      & = \frac{1}{h^\alpha}\Big(\frac{\alpha-2q}{2(p-q)}(1-\mathrm{e}^{-i\omega h})^\alpha
        \mathrm{e}^{i\omega hp}+\frac{2p-\alpha}{2(p-q)}(1-\mathrm{e}^{-i\omega h})^\alpha
        \mathrm{e}^{i\omega hq}\Big)\hat{u}(\omega)\\
      & = (i\omega)^\alpha\Big(\frac{\alpha-2q}{2(p-q)}W_p(i\omega h)
        +\frac{2p-\alpha}{2(p-q)}W_q(i\omega h)\Big)\hat{u}(\omega),
    \end{split}
  \end{equation}
  where
  \begin{equation}\label{eq:2.13}
    W_r(z)=\Big(\frac{1-\mathrm{e}^{-z}}{z}\Big)^\alpha\mathrm{e}^{rz}=1+(r-\frac{\alpha}{2})z+O(z^2),~r=p,q.
  \end{equation}
  Denoting $\hat{\phi}(\omega,h)=\mathscr{F}[{_L}\mathcal{D}_{h,p,q}^{\alpha}u](\omega)-\mathscr{F}[{_{-\infty}}D_x^{\alpha}u](\omega)$, then from \eqref{eq:2.12} and \eqref{eq:2.13}
  there exists
  \begin{equation}
    |\hat{\phi}(\omega,h)|\le Ch^2|i\omega|^{\alpha+2}|\hat{u}(\omega)|.
  \end{equation}
  With the condition $\mathscr{F}[{_{-\infty}}D_x^{\alpha+2}u](\omega)\in L^1(\mathbb{R})$, it yields
  \begin{equation}
    |{_L}\mathcal{D}_{h,p,q}^{\alpha}u-{_{-\infty}}D_x^{\alpha}u|=|\phi|\le
    \frac{1}{2\pi}\int_{\mathbb{R}}|\hat{\phi}(\omega,h)|
    \le C\|\mathscr{F}[{_{-\infty}}D_x^{\alpha+2}u](\omega)\|_{L^1}h^2=O(h^2).
  \end{equation}
\end{proof}
\begin{remark}
  For the right Riemann-Liouville fractional derivative, similar to
  Theorem \ref{thm:1}, we can check that
  \begin{equation}\label{eq:2.5}
    {_R}\mathcal{D}_{h,p,q}^{\alpha}u(x)=\frac{\alpha-2q}{2(p-q)}B_{h,p}^{\alpha}u(x)
    +\frac{2p-\alpha}{2(p-q)}B_{h,q}^{\alpha}f(x)={_x}D_{\infty}^{\alpha}u(x)+O(h^2),
  \end{equation}
  uniformly for $x\in\mathbb{R}$ under the conditions that $u\in L^1(\mathbb{R})$, ${_x}D_{\infty}^{\alpha+2}u$ and its Fourier transform belong to $L^1(\mathbb{R})$, where $p,q$ are integers and
  \begin{equation}\label{eq:2.14}
    B_{h,r}^{\alpha}u(x)=\frac{1}{h^\alpha}\sum_{k=0}^{\infty}g_k^{(\alpha)}u(x+(k-r)h).
  \end{equation}
\end{remark}
\begin{remark}
  Considering a well defined function $u(x)$ on the bounded interval $[a,b]$, if $u(a)=0$ or $u(b)=0$, the function $u(x)$ can be zero extended for $x<a$ or $x>b$. And then the $\alpha$ order left and right Riemann-Liouville fractional derivatives of $u(x)$ at each point $x$ can be approximated by the WSGD operators with second order accuracy
  \begin{equation}
    \begin{split}\label{eq:2.6}
      & _aD_x^{\alpha}u(x)=\frac{\lambda_1}{h^\alpha}\sum_{k=0}^{[\frac{x-a}{h}]+p}g_k^{(\alpha)}u(x-(k-p)h)
      +\frac{\lambda_2}{h^\alpha}\sum_{k=0}^{[\frac{x-a}{h}]+q}g_k^{(\alpha)}u(x-(k-q)h)+O(h^2), \\
      & _xD_b^{\alpha}u(x)=\frac{\lambda_1}{h^\alpha}\sum_{k=0}^{[\frac{b-x}{h}]+p}g_k^{(\alpha)}u(x+(k-p)h)
      +\frac{\lambda_2}{h^\alpha}\sum_{k=0}^{[\frac{b-x}{h}]+q}g_k^{(\alpha)}u(x+(k-q)h)+O(h^2),
    \end{split}
  \end{equation}
  where $\lambda_1=\frac{\alpha-2q}{2(p-q)},~\lambda_2=\frac{2p-\alpha}{2(p-q)}$.
\end{remark}
\begin{remark}\label{rem:3}
  The integers $p,q$ are the numbers of the points located on the right/left hand of the point $x$ used for evaluating the $\alpha$ order left/right Riemann-Liouville fractional derivatives at $x$,
  thus, when employing the difference method with \eqref{eq:2.6} for approximating non-periodic fractional differential equations on bounded interval,
  $p,q$ should be chosen satisfying $|p|\le1,|q|\le1$ to ensure that the nodes at which the values of $u$ needed in \eqref{eq:2.6} are within the bounded interval; otherwise, we need to use another way to discretize the fractional derivative when $x$ is close to the
  right/left boundary. When $(p,q)=(0,-1)$, the approximation method turns out to be unstable for time dependent problems. So two sets of $(p,q)$ can be selected to establish the difference scheme for fractional diffusion equations, that is $(1,0)$, $(1,-1)$, and the corresponding weights in \eqref{eq:2.3} and \eqref{eq:2.5} are $(\frac{\alpha}{2},\frac{2-\alpha}{2})$ and $(\frac{2+\alpha}{4},\frac{2-\alpha}{4})$.
  For $\alpha=2$, the WSGD operator \eqref{eq:2.3} is the centered difference approximation of second order derivative when $(p,q)$ equals to $(1,0)$ or $(1,-1)$;
  for $\alpha=1,~(p,q)=(1,0)$, the centered difference scheme for first order derivative is recovered.
\end{remark}
The simplified forms of the discreted approximations \eqref{eq:2.6} for Riemann-Liouville fractional derivatives with $(p,q)=(1,0)$, $(1,-1)$ are
\begin{equation}
  \begin{split}\label{eq:2.7}
    & _aD_x^{\alpha}u(x_i)=\frac{1}{h^\alpha}\sum_{k=0}^{i+1}w_k^{(\alpha)}u(x_{i-k+1})+O(h^2), \\
    & _xD_b^{\alpha}u(x_i)=\frac{1}{h^\alpha}\sum_{k=0}^{N-i+1}w_k^{(\alpha)}u(x_{i+k-1})+O(h^2),
  \end{split}
\end{equation}
where
\begin{equation}\left\{
  \begin{split}\label{eq:2.8}
    (p,q)=(1,0), \quad & w_0^{(\alpha)}=\frac{\alpha}{2}g_0^{(\alpha)},~ w_k^{(\alpha)}=\frac{\alpha}{2}g_k^{(\alpha)}
    +\frac{2-\alpha}{2}g_{k-1}^{(\alpha)},~k\ge1;\\
    (p,q)=(1,-1), \quad & w_0^{(\alpha)}=\frac{2+\alpha}{4}g_0^{(\alpha)},~w_1^{(\alpha)}=\frac{2+\alpha}{4}g_1^{(\alpha)},\\
    & w_k^{(\alpha)}=\frac{2+\alpha}{4}g_k^{(\alpha)}
    +\frac{2-\alpha}{4}g_{k-2}^{(\alpha)},~k\ge2.
  \end{split}\right.
\end{equation}
With Lemma \ref{lem:1} and some calculations, we obtain the
properties of the coefficients $w_k^{(\alpha)}$ in \eqref{eq:2.7}
corresponding to $(p,q)=(1,0),(1,-1)$ as follows.
\begin{lemma}\label{lem:2}
  The coefficients in \eqref{eq:2.7} satisfy the following properties for
  $1<\alpha\le2$,
  \begin{itemize}
    \item[(1)] if $(p,q)=(1,0)$,
      \begin{equation}\left\{
        \begin{split}
          & w_0^{(\alpha)}=\frac{\alpha}{2},~ w_1^{(\alpha)}=\frac{2-\alpha-\alpha^2}{2}<0, ~
            w_2^{(\alpha)}=\frac{\alpha(\alpha^2+\alpha-4)}{4},\\
          & 1\ge w_0^{(\alpha)}\ge w_3^{(\alpha)}\ge w_4^{(\alpha)}\ge\ldots\ge0,\\
          & \sum_{k=0}^\infty w_k^{(\alpha)}=0,~\sum_{k=0}^m
          w_k^{(\alpha)}<0,~m\ge2;
        \end{split}\right.
      \end{equation}
    \item[(2)] if $(p,q)=(1,-1)$,
      \begin{equation}\left\{
        \begin{split}
          & w_0^{(\alpha)}=\frac{2+\alpha}{4},~ w_1^{(\alpha)}=-\frac{2\alpha+\alpha^2}{4}<0, \\
          & w_2^{(\alpha)}=\frac{\alpha^3+\alpha^2-4\alpha+4}{8}>0,~
            w_3^{(\alpha)}=\frac{\alpha(2-\alpha)(\alpha^2+\alpha-8)}{6}\le0,\\
          & 1\ge w_0^{(\alpha)}\ge w_2^{(\alpha)}\ge w_4^{(\alpha)}\ge w_5^{(\alpha)}\ge\ldots\ge0,\\
          & \sum_{k=0}^\infty w_k^{(\alpha)}=0,~\sum_{k=0}^m w_k^{(\alpha)}<0,~m=1\,\,{\rm or }\,\, m\ge3.
        \end{split}\right.
      \end{equation}
  \end{itemize}
\end{lemma}
Next, we will explore the properties of the eigenvalues of the
difference matrix of \eqref{eq:2.7} on grid points
$\{x_k=a+kh,h=(b-a)/n,k=1,2,\ldots,n-1\}$. In the following, we
denote by $H$ the symmetric (respectively, hermitian) part of $A$ if
A is real (respectively, complex) matrix.
\begin{lemma}[\cite{Quarteroni:07}]\label{lem3}
  A real matrix $A$ of order $n$ is positive definite if and only if its symmetric part $H=\frac{A+A^T}{2}$ is positive definite; $H$ is positive definite if and only if the eigenvalues of $H$ are positive.
\end{lemma}
\begin{lemma}[\cite{Quarteroni:07}]\label{thm:2}
  If $A\in\mathbb{C}^{n\times n}$, let $H=\frac{A+A^*}{2}$
  be the hermitian part of $A$, $A^*$ the conjugate transpose of $A$, then for any eigenvalue $\lambda$ of $A$,
  there exists
  \begin{equation*}
    \lambda_{\min}(H)\le\mathrm{Re}(\lambda)\le\lambda_{\max}(H),
  \end{equation*}
  where $\mathrm{Re(\lambda)}$ represents the real part of $\lambda$, and $\lambda_{\min}(H)$,
  $\lambda_{\max}(H)$ are the minimum and maximum of the eigenvalues of $H$.
\end{lemma}

\begin{definition}[\cite{Chan:07}]
 Let Toeplitz matrix $T_n$ be of the following form,
  \begin{equation*}
    T_n=
    \begin{pmatrix}
      t_0     & t_{-1}  & \cdots & t_{2-n} & t_{1-n} \\
      t_1     & t_0     & t_{-1} & \cdots  & t_{2-n} \\
      \vdots  & t_1     & t_0    & \ddots  & \vdots  \\
      t_{n-2} & \cdots  & \ddots & \ddots  & t_{-1} \\
      t_{n-1} & t_{n-2} & \cdots & t_1     & t_0 \\
    \end{pmatrix},
  \end{equation*}
if the diagonals $\{t_k\}_{k=-n+1}^{n-1}$ are the Fourier
coefficients of a function $f$, i.e.,
 \begin{equation*} t_k=\frac{1}{2 \pi} \int_{-\pi}^{\pi} f(x)e^{-ikx}dx,\end{equation*} then the function $f$
is called the generating function of $T_n$.
\end{definition}

\begin{lemma}[Grenander-Szeg\"{o} theorem \cite{Chan:07,Chan:91}]\label{thm:3}
  For the above Toeplitz matrix $T_n$, if $f$ is a $2\pi$-periodic continuous real-valued function defined on $[-\pi,\pi]$, denote $\lambda_{\min}(T_n)$ and $\lambda_{\max}(T_n)$ as the smallest and largest eigenvalues of $T_n$, respectively.
  Then we have
  \begin{equation*}
    f_{\min}\le\lambda_{\min}(T_n)\le\lambda_{\max}(T_n)\le f_{\max},
  \end{equation*}
  where $f_{\min},~f_{\max}$ denote the minimum and maximum values of $f(x)$.
  Moreover, if $f_{\min}<f_{\max}$, then all eigenvalues of $T_n$ satisfy
  \begin{equation*}
    f_{\min}<\lambda(T_n)<f_{\max},
  \end{equation*}
  for all $n>0$; and furthermore if $f_{\min} \ge 0$, then $T_n$ is positive definite.
\end{lemma}
\begin{theorem}\label{thm:4}
  Let matrix $A$ be of the following form,
  \begin{equation}\label{eq:2.9}
    A=
    \begin{pmatrix}
      w_1^{(\alpha)}     & w_0^{(\alpha)}     &                &                &  \\
      w_2^{(\alpha)}     & w_1^{(\alpha)}     & w_0^{(\alpha)} &                &  \\
      \vdots             & w_2^{(\alpha)}     & w_1^{(\alpha)} & \ddots         &  \\
      w_{n-2}^{(\alpha)} & \cdots             & \ddots         & \ddots         & w_0^{(\alpha)} \\
      w_{n-1}^{(\alpha)} & w_{n-2}^{(\alpha)} & \cdots         & w_2^{(\alpha)} & w_1^{(\alpha)} \\
    \end{pmatrix},
  \end{equation}
  where the diagonals $\{w_k^{(\alpha)}\}_{k=0}^{n-1}$ are the coefficients given in \eqref{eq:2.7} corresponding to $(p,q)=(1,0)$ or $(1,-1)$. Then we have that any eigenvalue $\lambda$ of $A$ satisfies
  \begin{itemize}
    \item[(1)] $\mathrm{Re}(\lambda)\equiv0$, for $(p,q)=(1,0)$, $\alpha=1$,
    \item[(2)] $\mathrm{Re}(\lambda)<0$, for $(p,q)=(1,0)$, $1<\alpha\le2$,
    \item[(3)] $\mathrm{Re}(\lambda)<0$, for $(p,q)=(1,-1)$, $1\le\alpha\le2$.
  \end{itemize}
  Moreover, when $1<\alpha\le2$, matrix $A$ is negative definite, and the real parts of the eigenvalues of matrix $c_1A+c_2A^{\mathrm{T}}$ are less than 0, where $c_1,c_2\ge0,c_1^2+c_2^2\neq0$.
\end{theorem}
\begin{proof}
  We consider the symmetric part of matrix $A$, denoted as $H=\frac{A+A^T}{2}$. The generating functions of $A$ and $A^T$ are
  \begin{equation*}
    f_A(x)=\sum_{k=0}^{\infty}w_k^{(\alpha)}\mathrm{e}^{i(k-1)x},\quad
    f_{A^T}(x)=\sum_{k=0}^{\infty}w_k^{(\alpha)}\mathrm{e}^{-i(k-1)x},
  \end{equation*}
  respectively. Then $f(\alpha;x)=\frac{f_A(x)+f_{A^T}(x)}{2}$ is the generating function of
  $H$, and $f(\alpha;x)$ is a periodic continuous real-valued function on
  $[-\pi,\pi]$ since $f_A(x)$ and $f_{A^T}(x)$ are mutually
  conjugated.

  % since $f_A(x)$ and $f_{A^T}(x)$ are mutually conjugated, so $f(\alpha;x)$ is a periodic continuous real-valued function on $[-\pi,\pi]$.

  Case $(p,q)=(1,0)$: with the corresponding coefficients $w_k^{(\alpha)}$ given by \eqref{eq:2.8}, then
  \begin{equation*}
    \begin{split}
      f(\alpha;x) & =\frac{1}{2}\Big(\sum_{k=0}^{\infty}w_k^{(\alpha)}\mathrm{e}^{i(k-1)x}
              +\sum_{k=0}^{\infty}w_k^{(\alpha)}\mathrm{e}^{-i(k-1)x}\Big) \\
           & =\frac{1}{2}\Big(\frac{\alpha}{2}\mathrm{e}^{-ix}\sum_{k=0}^{\infty}g_k^{(\alpha)}\mathrm{e}^{ikx}
              +\frac{2-\alpha}{2}\sum_{k=0}^{\infty}g_k^{(\alpha)}\mathrm{e}^{ikx}
              +\frac{\alpha}{2}\mathrm{e}^{ix}\sum_{k=0}^{\infty}g_k^{(\alpha)}\mathrm{e}^{-ikx}
              +\frac{2-\alpha}{2}\sum_{k=0}^{\infty}g_k^{(\alpha)}\mathrm{e}^{-ikx}\Big) \\
           & =\frac{\alpha}{4}\Big(\mathrm{e}^{-ix}(1-\mathrm{e}^{ix})^\alpha
              +\mathrm{e}^{ix}(1-\mathrm{e}^{-ix})^\alpha\Big)
              +\frac{2-\alpha}{4}\Big((1-\mathrm{e}^{ix})^\alpha+(1-\mathrm{e}^{-ix})^\alpha\Big).
    \end{split}
  \end{equation*}
  Next we check $f(\alpha;x)\le0$ for $1<\alpha\le2$. Since $f(\alpha;x)$ is a real-valued and even function, we just consider its principal value on $[0,\pi]$. By the formula
  \begin{equation*}
    \mathrm{e}^{i\theta}-\mathrm{e}^{i\phi}=2i
    \sin\big(\frac{\theta-\phi}{2}\big)\mathrm{\mathrm{e}}^{\frac{i(\theta+\phi)}{2}},
  \end{equation*}
  we obtain
  \begin{equation}
    f(\alpha;x)=\big(2\sin(\frac{x}{2})\big)^\alpha\ \Big(\frac{\alpha}{2}\cos\big(\frac{\alpha}{2}(x-\pi)-x\big)
         +\frac{2-\alpha}{2}\cos\big(\frac{\alpha}{2}(x-\pi)\big)\Big).
  \end{equation}
  It is easy to prove that $f(\alpha;x)$ decreases with respect to $\alpha$, then $f(\alpha;x)\le f(1;x)\equiv0$; by Lemma \ref{thm:2} and \ref{thm:3}, $\mathrm{Re}(\lambda)\equiv0$ for $\alpha=1$, and $f(\alpha;x)$ is not identically zero for $1<\alpha\le2$, then we get $\mathrm{Re}(\lambda)<0$.

  Case $(p,q)=(1,-1)$: the corresponding generating function $f(\alpha;x)$ of $\frac{A+A^T}{2}$ can be calculated in the following form with coefficients $w_k^{(\alpha)}$ given by
  \eqref{eq:2.8},
  \begin{equation*}
    \begin{split}
      f(\alpha;x) & =\frac{1}{2}\Big(\sum_{k=0}^{\infty}w_k^{(\alpha)}\mathrm{e}^{i(k-1)x}
              +\sum_{k=0}^{\infty}w_k^{(\alpha)}\mathrm{e}^{-i(k-1)x}\Big) \\
           & =\frac{2+\alpha}{8}\Big(\mathrm{e}^{-ix}\sum_{k=0}^{\infty}g_k^{(\alpha)}\mathrm{e}^{ikx}
              +\mathrm{e}^{ix}\sum_{k=0}^{\infty}g_k^{(\alpha)}\mathrm{e}^{-ikx}\Big)
              +\frac{2-\alpha}{8}\Big(\mathrm{e}^{ix}\sum_{k=0}^{\infty}g_k^{(\alpha)}\mathrm{e}^{ikx}
              +\mathrm{e}^{-ix}\sum_{k=0}^{\infty}g_k^{(\alpha)}\mathrm{e}^{-ikx}\Big) \\
           & =\frac{2+\alpha}{8}\Big(\mathrm{e}^{-ix}(1-\mathrm{e}^{ix})^\alpha
              +\mathrm{e}^{ix}(1-\mathrm{e}^{-ix})^\alpha\Big)
              +\frac{2-\alpha}{8}\Big(\mathrm{e}^{ix}(1-\mathrm{e}^{ix})^\alpha
              +\mathrm{e}^{-ix}(1-\mathrm{e}^{-ix})^\alpha\Big).
    \end{split}
  \end{equation*}
  Next we check $f(\alpha;x)\le0$ for $1<\alpha\le2$. Since $f(\alpha;x)$ is a real-valued and even function, we just consider its principal value on $[0,\pi]$. By simple calculation, we obtain
  \begin{equation}
    f(\alpha;x)=\big(2\sin(\frac{x}{2})\big)^\alpha\ \Big(\frac{\alpha}{2}\sin\big(\frac{\alpha}{2}(x-\pi)\big)\sin(x)
         +\cos\big(\frac{\alpha}{2}(x-\pi)\big)\cos(x)\Big).
  \end{equation}
  We can also check that $f(\alpha;x)$ decreases with respect to $\alpha$, then $f(\alpha;x)\le f(1;x)=-2\sin^4(\frac{x}{2})\le0$, then by Lemma \autoref{thm:2} and \ref{thm:3}, we get $\mathrm{Re}(\lambda)<0$ for $1\le\alpha\le2$.

  From the above discussions and Lemma \ref{thm:3}, we know, for $1<\alpha\le2$, the matrix $\frac{1}{2}(A+A^T)$ is negative definite, which implies matrix $A$ is negative definite by Lemma \ref{lem3}.
  And the symmetric part of matrix $c_1A+c_2A^{\mathrm{T}}$ is $\frac{c_1+c_2}{2}(A+A^{\mathrm{T}})$, thus we obtain $\mathrm{Re}(\lambda(c_1A+c_2A^{\mathrm{T}}))<0$ for $1<\alpha\le2$.
\end{proof}
\begin{remark}
  For the case $(p,q)=(1,0)$ and $1<\alpha\le2$, we can check that the symmetric part $H$ of matrix $A$ in \eqref{eq:2.9} is strictly diagonally dominant by using Lemma \ref{lem:2}, and the elements of the main diagonal of $H$ are negative, then the eigenvalues of $H$ are less than zero by the Gershgorin circle theorem (\cite{Quarteroni:07},P188), therefore, with Lemma \ref{lem3} and \ref{thm:2}, we can also get $\mathrm{Re(\lambda(A))}<0$, and $A$ is negative definite.
\end{remark}

\begin{remark}
  By the same approach described in Theorem \ref{thm:4}, we can verify that the generating function of the symmetric part of difference matrix for $(p,q)=(0,-1)$ is not identically negative when $1<\alpha\le2$, which leads to the instability of the difference method to fractional diffusion equations for the same reason in the stability analysis in the sequel.
\end{remark}

%-----------------------------------------------------------------------------------------
\subsection{Third Order Approximations}
Similar to the second order approximations for Riemann-Liouville
fractional derivatives, we give a combination of three shifted
Gr\"{u}nwald difference operators
\begin{equation}\label{eq:2.10}
    {_L}\mathcal{G}_{h,p,q,r}^{\alpha}u(x)=\lambda_1A_{h,p}^{\alpha}u(x)
    +\lambda_2A_{h,q}^{\alpha}u(x)+\lambda_3A_{h,r}^{\alpha}u(x),
\end{equation}
where $p,q,r$ are integers and mutually non-equal, and
\begin{equation}
  \begin{split}\label{eq:2.16}
      \lambda_1=\frac{12qr-(6q+6r+1)\alpha+3\alpha^2}{12(qr-pq-pr+p^2)},\\
      \lambda_2=\frac{12pr-(6p+6r+1)\alpha+3\alpha^2}{12(pr-pq-qr+q^2)},\\
      \lambda_3=\frac{12pq-(6p+6q+1)\alpha+3\alpha^2}{12(pq-pr-qr+r^2)}.
  \end{split}
\end{equation}
Assuming $u\in L^1(\mathbb{R})$, and taking Fourier transform on \eqref{eq:2.10}, we get
\begin{equation}
  \begin{split}
      \mathscr{F}[{_L}\mathcal{G}_{h,p,q,r}^{\alpha}u](\omega)
      & = (i\omega)^\alpha\Big(\lambda_1W_p(i\omega h)+\lambda_2W_q(i\omega h)+\lambda_3W_r(i\omega h)
        \Big)\hat{u}(\omega) \\
      & = (i\omega)^\alpha\Big(1+C(i\omega h)^3\Big)\hat{u}(\omega),
  \end{split}
\end{equation}
where $W_s(z)$ is defined in \eqref{eq:2.13}.
If ${_{-\infty}}D_x^{\alpha+3}u$ and its Fourier transform belong to $L^1(\mathbb{R})$, then we have
\begin{equation}
  \begin{split}
    \big|{_L}\mathcal{G}_{h,p,q,r}^{\alpha}u-{_{-\infty}}D_x^{\alpha}u\big|
    & \le\frac{1}{2\pi}\int_{\mathbb{R}}
      \big|\mathscr{F}[{_L}\mathcal{G}_{h,p,q,r}^{\alpha}u-{_{-\infty}}D_x^{\alpha}u]\big| \\
    & \le C\|\mathscr{F}[{_{-\infty}}D_x^{\alpha+3}u](\omega)\|_{L^1}h^3=O(h^3).
  \end{split}
\end{equation}
The above results can be stated in the following theorem.
\begin{theorem}\label{thm:5}
  Let $u\in L^1(\mathbb{R})$, ${_{-\infty}}D_x^{\alpha+3}u$ and its Fourier transform belong to $L^1(\mathbb{R})$, and the following 3-WSGD operator \eqref{eq:2.10} satisfies
  \begin{equation}
    {_L}\mathcal{G}_{h,p,q,r}^{\alpha}u(x)={_{-\infty}}D_x^{\alpha}u(x)+O(h^3),
  \end{equation}
  uniformly for $x\in\mathbb{R}$.
\end{theorem}
If $u\in L^1(\mathbb{R})$, ${_x}D_{\infty}^{\alpha+3}u$ and its
Fourier transform belong to $L^1(\mathbb{R})$, we also have
\begin{equation}\label{eq:2.15}
    {_R}\mathcal{G}_{h,p,q,r}^{\alpha}u(x)=\lambda_1B_{h,p}^{\alpha}u(x)
    +\lambda_2B_{h,q}^{\alpha}u(x)+\lambda_3B_{h,r}^{\alpha}u(x)={_x}D_{\infty}^{\alpha}u+O(h^3),
\end{equation}
uniformly for $x\in\mathbb{R}$, where the operator $B_{h,s}^{\alpha}$ is given by \eqref{eq:2.14}, and $\lambda_i,i=1,2,3$ are the same as \eqref{eq:2.16}.

As stated in Remark \ref{rem:3}, the 3-WSGD operator can be utilized
for approximating Riemann-Liouville fractional differential
equations on bounded domain by finite difference method when
choosing $(p,q,r)=(1,0,-1)$, then the corresponding weight
coefficients in \eqref{eq:2.16} are
$\lambda_1=\frac{5}{24}\alpha+\frac{1}{8}\alpha^2,~\lambda_2=1+\frac{1}{12}\alpha-\frac{1}{4}\alpha^2,~\lambda_3=-\frac{7}{24}\alpha+\frac{1}{8}\alpha^2$.
For function $u(x)$ satisfying $u(a)=u(b)=0$ on grid points
$\{x_k=a+kh,h=(b-a)/n,k=1,\ldots,n-1\}$, the approximation matrix of
\eqref{eq:2.10} with $(p,q,r)=(1,0,-1)$ is
\begin{equation}
  \begin{split}\label{eq:2.17}
    G= & \lambda_1
    \begin{pmatrix}
      g_1^{(\alpha)}     & g_0^{(\alpha)}     &                &                &  \\
      g_2^{(\alpha)}     & g_1^{(\alpha)}     & g_0^{(\alpha)} &                &  \\
      \vdots             & g_2^{(\alpha)}     & g_1^{(\alpha)} & \ddots         &  \\
      g_{n-2}^{(\alpha)} & \cdots             & \ddots         & \ddots         & g_0^{(\alpha)} \\
      g_{n-1}^{(\alpha)} & g_{n-2}^{(\alpha)} & \cdots         & g_2^{(\alpha)} & g_1^{(\alpha)} \\
    \end{pmatrix}+\lambda_2
    \begin{pmatrix}
      g_0^{(\alpha)}     &                    &                &                &  \\
      g_1^{(\alpha)}     & g_0^{(\alpha)}     &                &                &  \\
      \vdots             & g_1^{(\alpha)}     & g_0^{(\alpha)} &                &  \\
      g_{n-3}^{(\alpha)} & \cdots             & \ddots         & \ddots         &                \\
      g_{n-2}^{(\alpha)} & g_{n-3}^{(\alpha)} & \cdots         & g_1^{(\alpha)} & g_0^{(\alpha)} \\
    \end{pmatrix} \\
     & +\lambda_3
    \begin{pmatrix}
      0                  &                    &                &                &  \\
      g_0^{(\alpha)}     & 0                  &                &                &  \\
      \vdots             & g_0^{(\alpha)}     & 0              &                &  \\
      g_{n-4}^{(\alpha)} & \cdots             & \ddots         & \ddots         &                \\
      g_{n-3}^{(\alpha)} & g_{n-4}^{(\alpha)} & \cdots         & g_0^{(\alpha)} & 0              \\
    \end{pmatrix}.
  \end{split}
\end{equation}
\begin{example}\label{exm:0}
  We utilize the approximation \eqref{eq:2.10} for simulating the steady state fractional diffusion problem
  \begin{equation}
    -{_0}D_x^{\alpha}u(x)=-\frac{\Gamma(3+\alpha)}{2}x^2,\quad x\in(0,1),
  \end{equation}
  with $u(0)=0,~u(1)=1$, and $1<\alpha<2$. The exact solution is $u(x)=x^{2+\alpha}$.
\end{example}
The 3-WSGD operator with $(p,q,r)=(1,0,-1)$ is utilized for
computing the solution of Example \ref{exm:0}, the numerical results
are given in Table \ref{tab:0}, from which the order and accuracy of
the 3-WSGD operator is verified.
\begin{table}[htp]\fontsize{10pt}{12pt}\selectfont
  \begin{center}%\def\tabcolsep{28.5pt}
  \caption{The maximum and $L^2$ errors and their convergence rates to Example \ref{exm:0} approximated by the 3-WSGD operator for $\alpha=1.1,1.9$.}\vspace{5pt}
  \begin{tabular*}{\linewidth}{@{\extracolsep{\fill}}*{1}{r}*{8}{c}}
    \toprule
     & \multicolumn{4}{c}{$\alpha=1.1$} & \multicolumn{4}{c}{$\alpha=1.9$} \\
    \cline{2-5} \cline{6-9} \\[-8pt]
    $N$ & $\|u^n-U^n\|_{\infty}$ & rate & $\|u^n-U^n\|$ & rate & $\|u^n-U^n\|_{\infty}$ & rate & $\|u^n-U^n\|$ & rate\\
    \midrule
    8 & 9.48629E-04 & - & 5.92003E-04 & - & 3.20333E-04 & - & 1.59788E-04 & - \\
    16 & 1.19530E-04 & 2.99  & 7.51799E-05 & 2.98 & 2.29262E-05 & 3.80  & 1.04858E-05 & 3.93 \\
    32 & 1.50130E-05 & 2.99  & 9.47995E-06 & 2.99 & 1.58500E-06 & 3.85  & 6.71546E-07 & 3.96 \\
    64 & 1.88094E-06 & 3.00  & 1.18999E-06 & 2.99 & 1.07818E-07 & 3.88  & 4.24776E-08 & 3.98 \\
    128 & 2.35382E-07 & 3.00  & 1.49052E-07 & 3.00 & 7.27733E-09 & 3.89  & 2.67067E-09 & 3.99 \\
    256 & 2.94392E-08 & 3.00  & 1.86501E-08 & 3.00 & 4.89318E-10 & 3.89  & 1.67325E-10 & 4.00 \\
    \toprule
  \end{tabular*}\label{tab:0}
  \end{center}
\end{table}
As we do in the above, the generating function of the symmetric part
$\frac{G+G^T}{2}$ of the Toeplitz matrix $G$ is
\begin{equation*}
  \begin{split}
    f(\alpha;x)= & \big(\frac{5}{48}\alpha+\frac{1}{16}\alpha^2\big)\Big(\mathrm{e}^{-ix}(1-\mathrm{e}^{ix})^\alpha
              +\mathrm{e}^{ix}(1-\mathrm{e}^{-ix})^\alpha\Big)\\
            & +\big(\frac{1}{2}+\frac{1}{24}\alpha-\frac{1}{8}\alpha^2\big)
              \Big((1-\mathrm{e}^{ix})^\alpha+(1-\mathrm{e}^{-ix})^\alpha\Big)\\
            & +\big(-\frac{7}{48}\alpha+\frac{1}{16}\alpha^2\big)\Big(\mathrm{e}^{ix}(1-\mathrm{e}^{ix})^\alpha
              +\mathrm{e}^{-ix}(1-\mathrm{e}^{-ix})^\alpha\Big),
  \end{split}
\end{equation*}
$x\in[-\pi,\pi]$. As matrix $\frac{G+G^T}{2}$ is symmetric, thus
$f(\alpha;x)$ is a real-valued and even function, so we consider it
on $[0,\pi]$ and get
\begin{equation*}
  \begin{split}
    f(\alpha;x)= & \big(2\sin(\frac{x}{2})\big)^\alpha\
              \Big(\big(\frac{5}{48}\alpha+\frac{1}{16}\alpha^2\big)\cos\big(\frac{\alpha}{2}(x-\pi)-x\big)
              +\big(\frac{1}{2}+\frac{1}{24}\alpha-\frac{1}{8}\alpha^2\big)\cos\big(\frac{\alpha}{2}(x-\pi)\big) \\
            & +\big(-\frac{7}{48}\alpha+\frac{1}{16}\alpha^2\big)\cos\big(\frac{\alpha}{2}(x-\pi)+x\big)\Big).
  \end{split}
\end{equation*}
We can check that $f(\alpha;x)$ is not identically positive or negative for $1\le\alpha\le2$, and the real parts of the eigenvalues of matrix $G$ are not always negative, so the finite difference scheme using \eqref{eq:2.10} or \eqref{eq:2.15} for time dependent fractional problems will not be unconditionally stable.
%=========================================================================================
\section{One Dimensional Space Fractional Diffusion Equation}
In this section, we consider the following two-sided one dimensional
space fractional diffusion equation
\begin{equation}\label{eq:3.1}
  \begin{cases}
    \frac{\partial u(x, t)}{\partial t}=K_1~{_a}D_x^{\alpha}u(x, t)+K_2~{_x}D_b^{\alpha}u(x, t)+f(x,t),  &   \text{$(x, t) \in (a, b)\times (0, T]$,} \\
    u(x,0)=u_0(x),    & \text{$x\in [a, b]$},\\
    u(a,t)=\phi_a(t),\ \ u(b,t)=\phi_b(t), &\text{$t\in [0, T]$},
  \end{cases}
\end{equation}
where both ${_a}D_x^{\alpha}$ and $_xD_b^{\alpha}$ are Riemann-Liouville fractional
operators with $1<\alpha\leq 2$. The diffusion coefficients $K_1$ and $K_2$ are
nonnegative constants with $K_1^2+K_2^2\neq0$. And if $K_1\neq0$, then
$\phi_a(t)\equiv0$; if $K_2\neq0$, then $\phi_b(t)\equiv0$. Next we will discretize the
problem \eqref{eq:3.1} by the second order accurate WSGD formula \eqref{eq:2.7}. In the
analysis of the numerical method that follows, we assume that \eqref{eq:3.1} has a unique
and sufficiently smooth solution.
%-----------------------------------------------------------------------------------------
\subsection{CN-WSGD scheme}
We partition the interval $[a, b]$ into a uniform mesh with the
space step $h=(b-a)/N$ and the time step $\tau=T/M$, where $N, M$
being two positive integers. And the set of grid points are denoted
by $x_i=ih$ and $t_n=n\tau$ for $1\leq i\leq N$ and $0\leq n\leq M$.
Let $t_{n+1/2}=(t_n+t_{n+1})/2$ for $0\leq n\leq M-1$, and we use
the following notations
\begin{equation*}
  u_i^n=u(x_i, t_n), \ \ f_i^{n+1/2}=f(x_i, t_{n+1/2}), \ \ \delta_t u_i^n=(u_i^{n+1}-u_i^n)/\tau.
\end{equation*}
Using the Crank-Nicolson technique for the time discretization of
\eqref{eq:3.1} leads to
\begin{equation*}
  \delta_t u_i^n-\frac{1}{2}\Big(K_1({_a}D_x^{\alpha}u)_i^n+K_1({_a}D_x^{\alpha}u)_i^{n+1}
  +K_2({_x}D_b^{\alpha}u)_i^n+K_2({_x}D_b^{\alpha}u)_i^{n+1} \Big)=f_i^{n+1/2}+O(\tau^2).
\end{equation*}
 In space discretization, we choose the WSGD operators ${_L}\mathcal{D}_{h,p,q}^{\alpha}u(x, t)$ and ${_R}\mathcal{D}_{h, p, q}^{\alpha}u(x,t)$ to approximate the Riemann-Liouville fractional derivatives ${_a}D_{x}^{\alpha}u(x,t)$ and ${_x}D_{b}^{\alpha}u(x,t)$ with second order accuracy, respectively, and $(p,q)=(1,0)$ or $(1,-1)$.
 This implies that
\begin{equation}
  \begin{split}\label{eq:3.2}
    &\delta_t u_i^n-\frac{1}{2}\Big(K_1~{_L}\mathcal{D}_{h,p,q}^{\alpha}u_i^n
      +K_1~{_L}\mathcal{D}_{h,p,q}^{\alpha}u_i^{n+1}
      +K_2~{_R}\mathcal{D}_{h,p,q}^{\alpha}u_i^n+K_2~{_R}\mathcal{D}_{h,p,q}^{\alpha}u_i^{n+1} \Big) \\
    &=f_i^{n+1/2}+\varepsilon_i^n,
  \end{split}
\end{equation}
where
\begin{equation}
   |\varepsilon_i^n|\le \tilde{c}(\tau^2+h^2).
\end{equation}
Multiplying \eqref{eq:3.2} by $\tau$ and separating the time layers,
we have
\begin{equation}
  \begin{split}
    &u_i^{n+1}-\frac{K_1\tau }{2}{_L}\mathcal{D}_{h,p,q}^{\alpha}u_i^{n+1}
      -\frac{K_2\tau }{2}{_R}\mathcal{D}_{h,p,q}^{\alpha}u_i^{n+1}   \\
    &=u_i^n+\frac{K_1\tau}{2}{_L}\mathcal{D}_{h,p,q}^{\alpha}u_i^n
      +\frac{K_2\tau}{2}{_R}\mathcal{D}_{h,p,q}^{\alpha}u_i^n+\tau f_i^{n+1/2}+O(\tau^3+\tau h^2).
  \end{split}
\end{equation}
Substituting ${_L}\mathcal{D}_{h,p,q}^{\alpha}u, {_R}\mathcal{D}_{h,p,q}^{\alpha}u$ by \eqref{eq:2.7}, we obtain that
\begin{equation}
  \begin{split}
    \label{eq:3.4}
      &u_i^{n+1}-\frac{K_1\tau}{2h^{\alpha}}\sum_{k=0}^{i+1}w_k^{(\alpha)}u_{i-k+1}^{n+1}
        -\frac{K_2\tau}{2h^{\alpha}}\sum_{k=0}^{N-i+1}w_k^{(\alpha)}u_{i+k-1}^{n+1} \\
      &=u_i^n+\frac{K_1\tau}{2h^{\alpha}}\sum_{k=0}^{i+1}w_k^{(\alpha)}u_{i-k+1}^n
        +\frac{K_2\tau}{2h^{\alpha}}\sum_{k=0}^{N-i+1}w_k^{(\alpha)}u_{i+k-1}^n+\tau f_i^{n+1/2}+O(\tau^3+\tau h^2).
  \end{split}
\end{equation}
Denoting $U_i^n$ as the numerical approximation of $u_i^n$, we derive the CN-WSGD scheme for \eqref{eq:3.1}
\begin{equation}
  \begin{split}
    \label{eq:3.5}
      &U_i^{n+1}-\frac{K_1\tau}{2h^{\alpha}}\sum_{k=0}^{i+1}w_k^{(\alpha)}U_{i-k+1}^{n+1}
        -\frac{K_2\tau}{2h^{\alpha}}\sum_{k=0}^{N-i+1}w_k^{(\alpha)}U_{i+k-1}^{n+1} \\
      &=U_i^n+\frac{K_1\tau}{2h^{\alpha}}\sum_{k=0}^{i+1}w_k^{(\alpha)}U_{i-k+1}^n
        +\frac{K_2\tau}{2h^{\alpha}}\sum_{k=0}^{N-i+1}w_k^{(\alpha)}U_{i+k-1}^n+\tau f_i^{n+1/2}.
  \end{split}
\end{equation}
For the convenience of implementation, using the matrix form of the
grid functions
\begin{equation*}
  U^n=\Big(U_1^n, U_2^n,\cdots, U_{N-1}^n\Big)^{\mathrm{T}}, \ \ \ F^n=\Big(f_1^{n+1/2}, f_2^{n+1/2}, \cdots,
  f_{N-1}^{n+1/2}\Big)^{\mathrm{T}},
\end{equation*}
makes the finite difference scheme \eqref{eq:3.5} be described as
\begin{equation}\label{eq:3.6}
  \Big(I-\frac{\tau}{2h^{\alpha}}(K_1A+K_2A^{\mathrm{T}})\Big)U^{n+1}
  =\Big(I+\frac{\tau}{2h^{\alpha}}(K_1A+K_2A^{\mathrm{T}})\Big)U^{n}+\tau F^n+H^n,
\end{equation}
where $A$ is given by \eqref{eq:2.9} and
\begin{equation}
 H^n=\frac{\tau}{2h^{\alpha}}
 \begin{bmatrix}
   K_1w_2^{(\alpha)}+K_2w_0^{(\alpha)}\\
   K_1w_3^{(\alpha)}\\
   \vdots \\
   K_1w_{N-1}^{(\alpha)}\\
   K_1w_N^{(\alpha)}
 \end{bmatrix}(U_{0}^n+U_{0}^{n+1})+
 \frac{\tau}{2h^{\alpha}}
 \begin{bmatrix}
   K_2w_N^{(\alpha)}\\
   K_2w_{N-1}^{(\alpha)}\\
   \vdots \\
   K_2w_3^{(\alpha)}\\
   K_1w_0^{(\alpha)}+K_2w_2^{(\alpha)}
 \end{bmatrix}(U_{N}^n+U_{N}^{n+1}).
\end{equation}
%-----------------------------------------------------------------------------------------
\subsection{Stability and Convergence}
Now we consider the stability and convergence analysis for the CN-WSGD scheme
\eqref{eq:3.6}. Define
\begin{equation*}
V_h=\{v:v=\{v_{i}\} \text{ is a grid function in } \{x_i=ih\}_{i=1}^{N-1}\text{~and~}v_{0}=v_{N}=0\}.
\end{equation*}
For any $v=\{v_{i}\}\in V_h$, we define its pointwise maximum norm
\begin{equation}
  \|v\|_{\infty}=\max\limits_{1\le i\le N-1} |v_{i}|
\end{equation}
and the following discrete norm
\begin{align*}
    &\|v\|=\sqrt{h\sum_{i=1}^{N-1}v_{i}^2}. %\ \ \ \|\delta_x v\|=\sqrt{h\sum_{i=0}^{N-1}(\delta_xv_{i})^2},
\end{align*}
%where $\delta_x v_i=\frac{v_{i+1}-v_i}{h}$.
\begin{theorem}
The finite difference scheme \eqref{eq:3.5} is unconditionally stable.
\end{theorem}
\begin{proof}
  Denoting $B=\frac{\tau}{2h^{\alpha}}(K_1A+K_2A^{\mathrm{T}})$. The matrix form of the difference approximation for problem \eqref{eq:3.1} can be rewritten as
  \begin{equation}
    (I-B)U^{n+1}=(I+B)U^{n}+\tau F^n+H^n.
  \end{equation}
  If denote $\lambda$ as an eigenvalue of matrix $B$, then $\frac{1+\lambda}{1-\lambda}$ is the eigenvalue of matrix $(I-B)^{-1}(I+B)$. The result of Theorem \ref{thm:4} shows that the eigenvalues of matrix $\frac{B+B^{\mathrm{T}}}{2}=\frac{\tau(K_1+K_2)}{4h^{\alpha}}(A+A^{\mathrm{T}})$ are negative, thus $\mathrm{Re(\lambda)}<0$, which implies that $|\frac{1+\lambda}{1-\lambda}|<1$. Therefore, the spectral radius of matrix $(I-B)^{-1}(I+B)$ is less than one, then the discreted scheme \eqref{eq:3.5} is unconditionally stable.
\end{proof}
\begin{remark}\label{rem:5}
  Considering the $\theta$ weighted scheme for the time discretization of \eqref{eq:3.1}, then the iterative matrix of the full discrete scheme is
  \begin{equation}\label{eq:3.7}
    \big(I-\theta B\big)^{-1}\big(I+(1-\theta)B\big),
  \end{equation}
  if $\lambda$ is an eigenvalue of matrix $B$, then the eigenvalue of \eqref{eq:3.7} is $\frac{1+(1-\theta)\lambda}{1-\theta\lambda}$. As $Re(\lambda)<0$, it is easy to check that
  \begin{equation}
    \Big|\frac{1+(1-\theta)\lambda}{1-\theta\lambda}\Big|<1
  \end{equation}
  for $\frac{1}{2}\le\theta\le1$. Then the $\theta$ weighted WSGD scheme for \eqref{eq:3.1} is unconditionally stable when $\frac{1}{2}\le\theta\le1$.
\end{remark}
Before verifying the unconditional convergence of the scheme
\eqref{eq:3.5}, we first present the discrete Gronwall's inequality.
%\begin{lemma}[\cite{Sun:05}]\label{lem8}
%  For any $v_i, (0\leq i\leq N)$ with $v_0=v_N=0$, we have
%  \begin{equation}
%    \|v\|^2\leq \frac{(b-a)^2}{6}\|\delta_x v\|^2, ~~~~~~\|v\|_{\infty}^2\leq\frac{b-a}{4}\|\delta_x v\|^2.
%  \end{equation}
%\end{lemma}
\begin{lemma}[\cite{Quarteroni:97}]\label{lem:3}
  Assume that $\{k_n\}$ and $\{p_n\}$ are nonnegative sequences, and the sequence $\{\phi_n\}$ satisfies
  \begin{equation*}
    \phi_0\leq g_0,\ \ \ \ \phi_n\leq g_0+\sum_{l=0}^{n-1}p_l+\sum_{l=0}^{n-1}k_l\phi_l,\ \ \ n\geq 1,
  \end{equation*}
  where $g_0\geq 0$. Then the sequence $\{\phi_n\}$ satisfies
  \begin{equation}
    \phi_n\leq \Big(g_0+\sum_{l=0}^{n-1}p_l\Big)\exp\Big(\sum_{l=0}^{n-1}k_l\Big),\ \ \ n\geq 1.
  \end{equation}
\end{lemma}
\begin{theorem}
  Let $u_i^n$ be the exact solution of problem \eqref{eq:3.1}, and $U_i^n$ the solution of the finite difference scheme \eqref{eq:3.5}, then for all $1\leq n\leq M$, we have
  \begin{equation}
    \|u^n-U^n\|\leq c(\tau^2+h^2),
  \end{equation}
  where $c$ denotes a positive constant and $\|\cdot\|$ stands for the discrete $L^2$-norm.
\end{theorem}
\begin{proof}
  Let $e_i^n=u_i^n-U_i^n$, and from \eqref{eq:3.4} and \eqref{eq:3.5}
  we have
  \begin{equation}\label{eq:3.8}
    (e^{n+1}-e^{n})-\frac{K_1\tau}{2h^{\alpha}}A(e^{n+1}+e^n)
    -\frac{K_2\tau}{2h^{\alpha}}A^{\mathrm{T}}(e^{n+1}+e^n)=\tau
    \varepsilon^n,
  \end{equation}
  where
  \begin{equation*}
    e^n=\Big(u_1^n-U_1^n,u_2^n-U_2^n,\cdots,u_{N-1}^n-U_{N-1}^n\Big)^{\mathrm{T}}, \varepsilon^n=\Big(\varepsilon_1^n,\varepsilon_2^n,\cdots,\varepsilon_{N-1}^n\Big)^{\mathrm{T}}.
  \end{equation*}
  Multiplying \eqref{eq:3.8} by $h$, and acting $(e^{n+1}+e^n)^\mathrm{T}$ on both sides, we obtain that
  \begin{equation*}
    \begin{split}
      & h(e^{n+1}+e^n)^{\mathrm{T}}I(e^{n+1}-e^{n})
        -\frac{K_1\tau}{2h^{\alpha-1}}(e^{n+1}+e^n)^{\mathrm{T}}A(e^{n+1}+e^n)    \\
      & -\frac{K_2\tau}{2h^{\alpha-1}}(e^{n+1}+e^n)^{\mathrm{T}}A^{\mathrm{T}}(e^{n+1}+e^n)
        =\tau h(e^{n+1}+e^n)^{\mathrm{T}}\varepsilon^n.
    \end{split}
  \end{equation*}
  By Theorem \ref{thm:4}, $A$ and its transpose $A^\mathrm{T}$ both being the negative definite matrices, we
  get
  \begin{equation}
   (e^{n+1}+e^n)^{\mathrm{T}}A(e^{n+1}+e^n)<0, \ \ \ \ (e^{n+1}+e^n)^{\mathrm{T}}A^{\mathrm{T}}(e^{n+1}+e^n)<0.
  \end{equation}
  And it yields that
  \begin{equation}
    h\sum_{i=1}^{N-1}\Big((e_i^{n+1})^2-(e_i^n)^2\Big)\leq \tau
    h\sum_{i=1}^{N-1}\Big(e_i^{n+1}+e_i^n\Big)\varepsilon_i^n.
  \end{equation}
  Summing up for all $0\leq k \leq n-1$, we have
  \begin{equation}
    \begin{split}
      h\sum_{i=1}^{N-1}(e_i^{n})^2&\leq \tau h\sum_{i=1}^{N-1}\sum_{k=0}^{n-1}\Big(e_i^{k+1}+e_i^{k}\Big)\varepsilon_i^k=
      \tau h\sum_{i=1}^{N-1}\sum_{k=1}^{n-1}e_i^k\Big(\varepsilon_i^{k-1}+\varepsilon_i^k\Big)+\tau h
      \sum_{i=1}^{N-1}e_i^{n}\varepsilon_i^{n-1}           \\
      &\leq \frac{\tau h}{2}\sum_{i=1}^{N-1}\sum_{k=1}^{n-1}(e_i^k)^2+\frac{\tau h}{2}\sum_{i=1}^{N-1}\sum_{k=1}^{n-1}\Big(\varepsilon_i^{k-1}+\varepsilon_i^k\Big)^2+\frac{h}{2}
      \sum_{i=1}^{N-1}(e_i^n)^2+\frac{h}{2}\sum_{i=1}^{N-1}(\tau \varepsilon_i^{n-1})^2.
    \end{split}
  \end{equation}
  By noting that $\varepsilon_i^n\leq \tilde{c}(\tau^2+h^2)$, and utilizing the discrete Gronwall's inequality, we obtain that
  \begin{equation}
    \|e^n\|^2\leq \tau\sum_{k=1}^{n-1}\|e^k\|^2+\Big(\tilde{c}(\tau^2+h^2)\Big)^2\leq\exp(T)\Big(\tilde{c}(\tau^2+h^2)\Big)^2\leq c\Big((\tau^2+h^2)^2\Big),
  \end{equation}
  which is the result that we need.
\end{proof}
\section{Two Dimensional Space Fractional Diffusion Equation}
We next consider the following two-sided space fractional diffusion
equation in two dimensions
\begin{equation}\label{eq:4.1}
  \begin{cases}
    \frac{\partial u(x, y, t)}{\partial t}=\Big(K_1^{+}{_a}D_x^{\alpha}u(x, y, t)+K_2^{+}{_x}D_b^{\alpha}u(x, y, t)\Big)\\
    ~~~~~~~~~~~~~~~~~~  +\Big(K_1^{-}{_c}D_y^{\beta}u(x, y, t)+K_2^{-}{_y}D_d^{\beta}u(x, y, t)\Big)+f(x, y, t),  &   \text{$(x, y, t) \in \Omega\times [0, T]$,} \\
    u(x, y, 0)=u_0(x, y),    & \text{$(x, y)\in \Omega$},\\
    u(x, y, t)=\varphi(x, y, t), &\text{$(x, y, t)\in\partial\Omega\times[0, T]$},
  \end{cases}
\end{equation}
where $\Omega=(a,b)\times(c,d)$, ${_a}D_x^{\alpha}, {_x}D_b^{\alpha}$ and
${_c}D_y^{\beta}, {_y}D_d^{\beta}$ are Riemann-Liouville fractional operators with
$1<\alpha, \beta \leq 2$. The diffusion coefficients satisfy $K_i^{+}, ~K_i^{-}\geq
0,~i=1,2$, $(K_1^{+})^2+(K_2^{+})^2\neq0$ and $(K_1^{-})^2+(K_2^{-})^2\neq0$. And the
boundary function $\varphi$ satisfies, if $K_1^{+}\neq0$, then $\varphi(a,y,t)=0$; if
$K_1^{-}\neq0$, then $\varphi(b,y,t)=0$; if $K_2^{+}\neq0$, then $\varphi(x,c,t)=0$; if
$K_2^{-}\neq0$, then $\varphi(x,d,t)=0$. We assume that \eqref{eq:4.1} has a unique and
sufficiently smooth solution.
%-----------------------------------------------------------------------------------------
\subsection{CN-WSGD scheme}
Now we establish the Crank-Nicolson difference scheme by using WSGD formula
\eqref{eq:2.7} for problem \eqref{eq:4.1}. We partition the domain $\Omega$ into a
uniform mesh with the space steps $h_x=(b-a)/N_x, h_y=(d-c)/N_y$ and the time step
$\tau=T/M$, where $N_x, N_y, M$ being positive integers. And the set of grid points
are denoted by $x_i=ih_x, y_j=jh_y$ and $t_n=n\tau$ for $1\leq i\leq N_x, 1\leq j\leq
N_y$ and $0\leq n\leq M$. Let $t_{n+1/2}=(t_n+t_{n+1})/2$ for $0\leq n\leq M-1$, and we
use the following notations
\begin{equation*}
  u_{i, j}^n=u(x_i, y_j, t_n), \ \ f_{i, j}^{n+1/2}=f(x_i, y_j, t_{n+1/2}), \ \ \delta_t u_{i, j}^n=(u_{i, j}^{n+1}-u_{i, j}^n)/\tau.
\end{equation*}
Discretizing \eqref{eq:4.1} in time direction leads to
\begin{equation}\label{eq:4.2}
  \begin{split}
    \delta_t u_{i,j}^{n}&=\frac{1}{2}\Big(K_1^{+}(_aD_x^{\alpha}u)_{i,j}^{n+1}+K_2^{+}(_xD_b^{\alpha}u)_{i,j}^{n+1}+K_1^{-}(_cD_y^{\beta}u)_{i,j}^{n+1}+K_2^{-}(_yD_d^{\beta}u)_{i,j}^{n+1}                           \\
    &+K_1^{+}(_aD_x^{\alpha}u)_{i,j}^n+K_2^{+}(_xD_b^{\alpha}u)_{i,j}^n+K_1^{-}(_cD_y^{\beta}u)_{i,j}^n+K_2^{-}(_yD_d^{\beta}u)_{i,j}^n\Big)+f_{i,j}^{n+1/2}+O(\tau^2).
  \end{split}
\end{equation}
In space discretization, we choose the WSGD operators
$_L\mathcal{D}_{h_{x}, p, q}^{\alpha}u, ~_R\mathcal{D}_{h_x, p,
q}^{\alpha}u$ and $_L\mathcal{D}_{h_{y}, p, q}^{\beta}u,~
_R\mathcal{D}_{h_y, p, q}^{\beta}u$ to respectively approximate the
fractional diffusion terms $_aD_x^{\alpha}u, ~_xD_b^{\alpha}u$ and
$_cD_y^{\beta}u,~ _yD_d^{\beta}u$. And multiplying \eqref{eq:4.2} by
$\tau$ and separating the time layers, we have that
\begin{equation}\label{eq:4.3}
  \begin{split}
    &\Big(1-\frac{K_1^{+}\tau}{2}{_L}\mathcal{D}_{h_x,p,q}^{\alpha}
        -\frac{K_2^{+}\tau}{2}{_R}\mathcal{D}_{h_x,p,q}^{\alpha}
        -\frac{K_1^{-}\tau}{2}{_L}\mathcal{D}_{h_y,p,q}^{\beta}
        -\frac{K_2^{-}\tau}{2}{_R}\mathcal{D}_{h_y,p,q}^{\beta}\Big)u_{i,j}^{n+1} \\
    &=\Big(1+\frac{K_1^{+}\tau}{2}{_L}\mathcal{D}_{h_x,p,q}^{\alpha}
        +\frac{K_2^{+}\tau}{2}{_R}\mathcal{D}_{h_x,p,q}^{\alpha}
        +\frac{K_1^{-}\tau}{2}{_L}\mathcal{D}_{h_y,p,q}^{\beta}
        +\frac{K_2^{-}\tau}{2}{_R}\mathcal{D}_{h_y,p,q}^{\beta}\Big)u_{i,j}^{n}
        +\tau f_{i,j}^{n+1/2}+\tau\varepsilon_{i,j}^{n},
  \end{split}
\end{equation}
where $|\varepsilon_{i,j}^{n}|\leq \tilde{c}(\tau^2+h^2)$ denotes the truncation error.
And we denote
\begin{equation*}
  \delta_x^\alpha=K_1^{+}{_L}\mathcal{D}_{h_x,p,q}^{\alpha}+K_2^{+}{_R}\mathcal{D}_{h_x,p,q}^{\alpha}, \qquad
  \delta_y^\beta=K_1^{-}{_L}\mathcal{D}_{h_y,p,q}^{\beta}+K_2^{-}{_R}\mathcal{D}_{h_y,p,q}^{\beta}.
\end{equation*}
For simplicity, the step sizes are chosen as the same, $h=h_x=h_y$.
Using the Taylor expansion, we have
\begin{equation}\label{eq:4.4}
 \frac{\tau^2}{4}\delta_x^{\alpha}\delta_x^{\beta}(u_{i,j}^{n+1}-u_{i,j}^{n})=\frac{\tau^3}{4}\Big((K_1^{+}{_a}D_x^{\alpha}+K_2^{+}{_x}D_{b}^{\alpha})(K_1^{-}{_c}D_{y}^{\beta}+K_2^{-}{_y}D_{d}^{\beta})u_t\Big)_{i,j}^{n+1/2}+O(\tau^5+\tau^3h^2).
 \end{equation}
Adding formula \eqref{eq:4.4} to the right-hand side of
\eqref{eq:4.3} and making the factorization leads to
\begin{equation}
  \begin{split}\label{eq:4.6}
    \Big(1-\frac{\tau}{2}\delta_x^\alpha\Big)\Big(1-\frac{\tau}{2}\delta_y^\beta\Big)u_{i,j}^{n+1}
    = \Big(1+\frac{\tau}{2}\delta_x^\alpha\Big)\Big(1+\frac{\tau}{2}\delta_y^\beta\Big)u_{i,j}^n+\tau f_{i,j}^{n+1/2}+\tau \varepsilon_{i,j}^{n}+O(\tau^3+\tau^3h^2).
  \end{split}
\end{equation}
Denoting by $U_{i,j}^n$ the numerical approximation to $u_{i,j}^n$,
we obtain the finite difference approximation for problem
\eqref{eq:4.1}
\begin{equation}
  \begin{split}\label{eq:4.7}
    & \Big(1-\frac{\tau}{2}\delta_x^\alpha\Big)\Big(1-\frac{\tau}{2}\delta_y^\beta\Big)U_{i,j}^{n+1}
    = \Big(1+\frac{\tau}{2}\delta_x^\alpha\Big)\Big(1+\frac{\tau}{2}\delta_y^\beta\Big)U_{i,j}^n+\tau f_{i,j}^{n+1/2}.
  \end{split}
\end{equation}
For efficiently solving (\ref{eq:4.7}), the following techniques can
be used. Peaceman-Rachford ADI scheme \cite{Sun:05}:
\begin{subequations}\label{eq:4.16}
\begin{align}
    &\Big(1-\frac{\tau}{2}\delta_x^\alpha\Big)V_{i,j}^{n}~~~
     =\Big(1+\frac{\tau}{2}\delta_y^\beta\Big)U_{i,j}^{n}+\frac{\tau}{2}f_{i,j}^{n+1/2}, \\
    &\Big(1-\frac{\tau}{2}\delta_y^\beta\Big)U_{i,j}^{n+1}
     =\Big(1+\frac{\tau}{2}\delta_x^\alpha\Big)V_{i,j}^{n}+\frac{\tau}{2}f_{i,j}^{n+1/2}.
\end{align}
\end{subequations}
Douglas ADI scheme \cite{Douglas:01}:
\begin{subequations}\label{eq:4.17}
\begin{align}
    &\Big(1-\frac{\tau}{2}\delta_x^\alpha\Big)V_{i,j}^{n}~~~=
     \Big(1+\frac{\tau}{2}\delta_x^\alpha+\tau\delta_y^\beta\Big)U_{i,j}^{n}+\tau f_{i,j}^{n+1/2}, \\
    &\Big(1-\frac{\tau}{2}\delta_y^\beta\Big)U_{i,j}^{n+1}
     =V_{i,j}^n-\frac{\tau}{2}\delta_y^\beta U_{i,j}^n.
\end{align}
\end{subequations}
D'Yakonov ADI scheme \cite{Sun:05}:
\begin{subequations}\label{eq:4.18}
\begin{align}
    &\Big(1-\frac{\tau}{2}\delta_x^\alpha\Big)V_{i,j}^{n}~~~=\Big(1+\frac{\tau}{2}\delta_x^\alpha\Big)
     \Big(1+\frac{\tau}{2}\delta_y^\beta\Big)U_{i,j}^{n}+\tau f_{i,j}^{n+1/2}, \\
    &\Big(1-\frac{\tau}{2}\delta_y^\beta\Big)U_{i,j}^{n+1}=V_{i,j}^n.
\end{align}
\end{subequations}
A simple calculation shows that
\begin{equation}
\label{eq:4.8}
\frac{\tau^3}{4}\delta_x^{\alpha}\delta_y^{\beta}f_{i,j}^{n+1/2}=\frac{\tau^3}{4}(K_1^{+}{_a}D_x^{\alpha}+K_2^{+}{_x}D_{b}^{\alpha})(K_1^{-}{_c}D_{y}^{\beta}+K_2^{-}{_y}D_{d}^{\beta})f_{i,j}^{n+1/2}+O(\tau^3h^2).
\end{equation}
Then from \eqref{eq:4.6} and \eqref{eq:4.8}, it yields that
\begin{equation}
  \begin{split}
    \Big(1-\frac{\tau}{2}\delta_x^\alpha\Big)\Big(1-\frac{\tau}{2}\delta_y^\beta\Big)u_{i,j}^{n+1}
    = \Big(1+\frac{\tau}{2}\delta_x^\alpha\Big)\Big(1+\frac{\tau}{2}\delta_y^\beta\Big)u_{i,j}^n+\tau f_{i,j}^{n+1/2}+
    \frac{\tau^3}{4}\delta_x^{\alpha}\delta_y^{\beta}f_{i,j}^{n+1/2}+\tau \tilde{\varepsilon}_{i,j}^{n}.
  \end{split}
\end{equation}
where
\begin{equation}
\tilde{\varepsilon}_{i,j}^n=\varepsilon_{i,j}^{n}-\frac{\tau^2}{4}(K_1^{+}{_a}D_x^{\alpha}+K_2^{+}{_x}D_{b}^{\alpha})(K_1^{-}{_c}D_{y}^{\beta}+K_2^{-}{_y}D_{d}^{\beta})f_{i,j}^{n+1/2}+O(\tau^2+\tau^2h^2).
\end{equation}
Eliminating the truncating error and denoting $U_{i,j}^n$ as the numerical approximation of $u_{i,j}^n$, we have
\begin{equation}
  \begin{split}\label{eq:4.10}
    \Big(1-\frac{\tau}{2}\delta_x^\alpha\Big)\Big(1-\frac{\tau}{2}\delta_y^\beta\Big)U_{i,j}^{n+1}
    = \Big(1+\frac{\tau}{2}\delta_x^\alpha\Big)\Big(1+\frac{\tau}{2}\delta_y^\beta\Big)U_{i,j}^n+\tau f_{i,j}^{n+1/2}+
    \frac{\tau^3}{4}\delta_x^{\alpha}\delta_y^{\beta}f_{i,j}^{n+1/2}.
  \end{split}
\end{equation}
Introducing the intermediate variable $V_{i,j}^{n}$, we obtain the locally one-dimensional (LOD) scheme mentioned in \cite{Qin:11,Wang:06},
\begin{subequations}\label{eq:4.15}
\begin{align}
    &\Big(1-\frac{\tau}{2}\delta_x^\alpha\Big)V_{i,j}^{n}~~~
      =\Big(1+\frac{\tau}{2}\delta_x^\alpha\Big)U_{i,j}^{n}
      +\frac{\tau}{2}\Big(1+\frac{\tau}{2}\delta_x^\alpha\Big)f_{i,j}^{n+1/2},\\
    &\Big(1-\frac{\tau}{2}\delta_y^\beta\Big)U_{i,j}^{n+1}
      =\Big(1+\frac{\tau}{2}\delta_y^\beta\Big)V_{i,j}^{n}
      +\frac{\tau}{2}\Big(1-\frac{\tau}{2}\delta_y^\beta\Big)f_{i,j}^{n+1/2}.
\end{align}
\end{subequations}
%---------------------------------------------------------------------------------
\subsection{Stability and Convergence}
Now we consider the stability and convergence analysis for the
CN-WSGD scheme \eqref{eq:4.7}. Define the sets of the index of the
interior and boundary mesh grid points in domain $[a,b]\times[c,d]$,
respectively, as
\begin{align*}
  &\Lambda_h=\{(i,j) : 1\leq i\leq N_x-1, 1\leq j\leq N_y-1\},             \\
  &\partial\Lambda_h=\{(i,j) : i=0,N_x; 0\leq j \leq N_y\} \cup\{(i, j) : 0\leq i \leq N_x; j=0,N_y\}.
\end{align*}
For any $v=\{v_{i}\}\in V_h$, we define its pointwise maximum norm
and discrete $L^2$ norm, respectively, as follows
\begin{equation}
  \|v\|_{\infty}=\max\limits_{(i,j)\in\Lambda_h} |v_{i,j}|, \quad \|v\|=\sqrt{h^2\sum_{i=1}^{N_x-1}\sum_{j=1}^{N_y-1}v_{i,j}^2},
\end{equation}
where
\begin{equation*}
  V_h=\{v:v=\{v_{i,j}\} \text{ is a grid function in } \Lambda_h \text{~and~}v_{i,j}=0 \text{ on } \partial\Lambda_h \}.
\end{equation*}
In the following, we list some properties of Kronecker products of matrices.
\begin{lemma}[\cite{Laub:05}]\label{lem:4}
  Let $A\in\mathbb{R}^{n\times n}$ have eigenvalues $\{\lambda_i\}_{i=1}^n$, and $B\in\mathbb{R}^{m\times m}$ have eigenvalues $\{\mu_j\}_{j=1}^m$. Then the $mn$ eigenvalues of $A\otimes B$, which represents the kronecker product of matrix $A$ and $B$, are
 \begin{equation*}
    \lambda_1\mu_1,\ldots,\lambda_1\mu_m,\lambda_2\mu_1,\ldots,\lambda_2\mu_m,\ldots,\lambda_n\mu_1,\ldots,\lambda_n\mu_m.
  \end{equation*}
\end{lemma}
\begin{lemma}[\cite{Laub:05}]\label{lem:5}
  Let $A\in \mathbb{R}^{m\times n}, B\in \mathbb{R}^{r\times s}, C\in \mathbb{R}^{n\times p}, D\in \mathbb{R}^{s\times t}$. Then
  \begin{equation}
    (A\otimes B)(C\otimes D)=AC\otimes BD~~(\in \mathbb{R}^{mr\times pt}).
  \end{equation}
  Moreover, if $A, B\in \mathbb{R}^{n\times n}$, $I$ is a unit matrix of order $n$, then matrices $I\otimes A$ and $B\otimes I$ commute.
\end{lemma}
\begin{lemma}[\cite{Laub:05}]\label{lem:6}
  For all $A$ and $B$, $(A\otimes B)^{\mathrm{T}}=A^{\mathrm{T}}\otimes B^{\mathrm{T}}$.
\end{lemma}
%\mred{
%\begin{lemma}[\cite{Zhang:11}]\label{lem:7}
%  Let $A, B$ to be two positive semidefinite matrices, symbolized $A\geq 0$ and $B\geq 0$. Then $A\otimes B\geq 0$.
%\end{lemma}
%}
In the theoretical analysis of the numerical method, we choose
$N=N_x=N_y$ for simplification.
\begin{theorem}\label{thm:6}
  The difference scheme \eqref{eq:4.7} is unconditionally stable for $1<\alpha,\beta\le2$.
\end{theorem}
\begin{proof}
  We represent the discrete functions $U_{i,j}^{n}$ and $f_{i,j}^{n+1/2}$  into vector forms with
  \begin{align*}
    & U^{n}=(u_{1,1}^{n},u_{2,1}^{n},\cdots,u_{N-1,1}^{n},u_{1,2}^{n},u_{2,2}^{n},\cdots,u_{N-1,2}^{n},\cdots,u_{1,N-1}^{n},u_{2,N-1}^n,\cdots,u_{N-1,N-1}^{n})^{\mathrm{T}},      \\
    & F^{n+1/2}=(f_{1,1}^{n+1/2},f_{2,1}^{n+1/2},\cdots,f_{N-1,1}^{n+1/2},f_{1,2}^{n+1/2},f_{2,2}^{n+1/2},\cdots,f_{N-1,2}^{n+1/2},\\
    & \hspace{5cm}\cdots,f_{1,N-1}^{n+1/2},f_{2,N-1}^{n+1/2},\cdots,f_{N-1,N-1}^{n+1/2})^{\mathrm{T}},
  \end{align*}
  and denote
  \begin{equation}\label{eq:4.13}
    \mathcal{D}_x=\frac{K_1^{+}\tau}{2h^{\alpha}}I\otimes A_{\alpha}+\frac{K_2^{+}\tau}{2h^{\alpha}}I\otimes A_{\alpha}^{\mathrm{T}},  ~~~~
    \mathcal{D}_y=\frac{K_1^{-}\tau}{2h^{\beta}}A_{\beta}\otimes I+\frac{K_2^{-}\tau}{2h^{\beta}}A_{\beta}^{\mathrm{T}}\otimes I,
  \end{equation}
  where the symbol $\otimes$ denotes the Kronecker product, $I$ is the unit matrix, and matrices $A_{\alpha}$ and $A_{\beta}$ are defined in \eqref{eq:2.9} corresponding to $\alpha,\beta$, respectively.
  Therefore, the difference scheme \eqref{eq:4.7} can be expressed as
  \begin{equation}
       \big(I-\mathcal{D}_x\big)
       \big(I-\mathcal{D}_y\big)U^{n+1}
      =\big(I+\mathcal{D}_x\big)
       \big(I+\mathcal{D}_y\big)U^{n}+\tau F^{n+1/2},
  \end{equation}
  then the relationship between the error $e^{n+1}$ in $U^{n+1}$ and the error $e^{n}$ in $U^{n}$ is given by
  \begin{equation}\label{eq:4.19}
      e^{n+1}=\big(I-\mathcal{D}_y\big)^{-1}\big(I-\mathcal{D}_x\big)^{-1}\big(I+\mathcal{D}_x\big)
       \big(I+\mathcal{D}_y\big)e^{n}.
  \end{equation}Using Lemma \ref{lem:5}, we can check that $\mathcal{D}_x$ and $\mathcal{D}_y$ commute, i.e.,
  \begin{equation}\label{eq:4.9}
    \mathcal{D}_x\mathcal{D}_y=\mathcal{D}_y\mathcal{D}_x=\frac{\tau^2}{4h^{\alpha+\beta}}(K_1^{-}A_{\beta}+K_2^{-}A_{\beta}^{\mathrm{T}})\otimes
    (K_1^{+}A_{\alpha}+K_2^{+}A_{\alpha}^{\mathrm{T}}).
  \end{equation}
  Thus \eqref{eq:4.19} can be rewritten as
  \begin{equation}
      e^{n}=\Big(\big(I-\mathcal{D}_y\big)^{-1}\big(I+\mathcal{D}_y\big)\Big)^n
      \Big(\big(I-\mathcal{D}_x\big)^{-1}\big(I+\mathcal{D}_x\big)\Big)^ne^{0}.
  \end{equation}
  We can also calculate the symmetric part of $\mathcal{D}_x$ by Lemma \ref{lem:6} as
  \begin{equation*}\label{eq:4.11}
    \frac{\mathcal{D}_x+\mathcal{D}_x^{\mathrm{T}}}{2}=\frac{(K_1^{+}+K_2^{+})\tau}{2h^\alpha}I\otimes\Big(\frac{A_{\alpha}+A_{\alpha}^{\mathrm{T}}}{2}\Big),\quad
    \frac{\mathcal{D}_y+\mathcal{D}_y^{\mathrm{T}}}{2}=\frac{(K_1^{-}+K_2^{-})\tau}{2h^\beta}I\otimes\Big(\frac{A_{\beta}+A_{\beta}^{\mathrm{T}}}{2}\Big).
  \end{equation*}
  And from Theorem \ref{thm:4}, the eigenvalues of $\frac{A_{\alpha}+A_{\alpha}^{\mathrm{T}}}{2}$ and $\frac{A_{\beta}+A_{\beta}^{\mathrm{T}}}{2}$ are all negative when $1<\alpha,\beta\le2$.
  Defining $\lambda_\alpha$ and $\lambda_\beta$ as an eigenvalue of matrices $\mathcal{D}_x$ and $\mathcal{D}_y$, respectively,
  then it yields from the consequences of Lemma \ref{thm:2} and \ref{lem:4} that the real parts of $\lambda_\alpha$ and $\lambda_\beta$ are both less than zero. Since $(1+\lambda_\alpha)/(1-\lambda_\alpha)$ and $(1+\lambda_\beta)/(1-\lambda_\beta)$ are eigenvalues of matrices $(I-\mathcal{D}_x)^{-1}(I+\mathcal{D}_x)$ and $(I-\mathcal{D}_y)^{-1}(I+\mathcal{D}_y)$, respectively, thus the spectral radius of each matrix is less than 1, which follows that $\big((I-\mathcal{D}_x)^{-1}(I+\mathcal{D}_x)\big)^n$ and $\big((I-\mathcal{D}_y)^{-1}(I+\mathcal{D}_y)\big)^n$ converge to zero matrix (see Theorem 1.5 in \cite{Quarteroni:07}). Therefore the difference scheme \eqref{eq:4.7} is unconditionally stable.
\end{proof}
\begin{remark}
  For the similar reason described in Remark \ref{rem:5} and the proof of  Theorem \ref{thm:6}, we conclude that the WSGD scheme with $\theta$ weighted scheme for the time discretization for \eqref{eq:4.1} is unconditionally stable when $\frac{1}{2}\le\theta\le1$.
\end{remark}
\begin{lemma}\label{lem:11}
  Let $\mathcal{D}_x$ and $\mathcal{D}_y$ be defined in \eqref{eq:4.13}, then
  \begin{align*}
    & \|(I-\mathcal{D}_x)^{-1}(I-\mathcal{D}_y)^{-1}\|_2\le 1,   \\
    & \|(I-\mathcal{D}_\gamma)^{-1}(I+\mathcal{D}_\gamma)\|_2\le 1, ~~\gamma=x,y,
  \end{align*}
  where $\|\cdot\|_2$ denotes the 2-norm (spectral norm).
\end{lemma}
\begin{proof}
    From Theorem \ref{thm:4} and Lemma \ref{lem:4}, we know that $\mathcal{D}_x+\mathcal{D}_x^\mathrm{T}$ and $\mathcal{D}_y+\mathcal{D}_y^\mathrm{T}$ are negative semi-definite and symmetric matrices. Then for any $v=(v_1,v_2,\cdots,v_n)^\mathrm{T}\in \mathbb{R}^n$, we obtain that
    \begin{equation*}
    v^\mathrm{T}v\le v^\mathrm{T}(I-\mathcal{D}_\gamma^\mathrm{T})(I-\mathcal{D}_\gamma)v,~~ \gamma=x, y.
    \end{equation*}
    Substituting $v$ and $v^\mathrm{T}$ by $(I-\mathcal{D}_\gamma)^{-1}v$ and $v^\mathrm{T}(I-\mathcal{D}_\gamma^\mathrm{T})^{-1}$, respectively, for any $v\in \mathbb{R}^n$, we get
    \begin{equation*}
      v^\mathrm{T}(I-\mathcal{D}_\gamma^\mathrm{T})^{-1}(I-\mathcal{D}_\gamma)^{-1}v\le v^\mathrm{T}v, ~~ \gamma=x, y.
    \end{equation*}
    Thus, it leads to
    \begin{equation*}
        \|(I-\mathcal{D}_\gamma)^{-1}\|_2=\sup_{v\neq 0}\frac{v^\mathrm{T}(I-\mathcal{D}_\gamma^\mathrm{T})^{-1}
        (I-\mathcal{D}_\gamma)^{-1}v}{v^\mathrm{T}v}\le1,~~\gamma=x, y.
    \end{equation*}
    Connsequently,
    \begin{equation*}
    \|(I-\mathcal{D}_x)^{-1}(I-\mathcal{D}_y)^{-1}\|_2\le\|(I-\mathcal{D}_x)^{-1}\|_2\|(I-\mathcal{D}_y)^{-1}\|_2\le 1
    \end{equation*}
    holds.

    Since $\mathcal{D}_x+\mathcal{D}_x^\mathrm{T}$ and $\mathcal{D}_y+\mathcal{D}_y^\mathrm{T}$ are negative semi-definite and symmetric, for any $v\in \mathbb{R}^n$, we have
    \begin{equation*}
        v^\mathrm{T}(I+\mathcal{D}_\gamma^\mathrm{T})(I+\mathcal{D}_\gamma)v\le
        v^\mathrm{T}(I-\mathcal{D}_\gamma^\mathrm{T})(I-\mathcal{D}_\gamma)v,~~\gamma=x,y.
    \end{equation*}
    By choosing vector $(I-\mathcal{D}_\gamma)^{-1}v$, we arrive at that for any $v\in \mathbb{R}^n$,
    \begin{equation*}
        v^\mathrm{T}(I-\mathcal{D}_\gamma^\mathrm{T})^{-1}(I+\mathcal{D}_\gamma^\mathrm{T})(I+\mathcal{D}_\gamma)
        (I-\mathcal{D}_\gamma)^{-1}v\le v^\mathrm{T}v,~~\gamma=x,y.
    \end{equation*}
    As $(I-\mathcal{D}_\gamma)^{-1}(I+\mathcal{D}_\gamma)=(I+\mathcal{D}_\gamma)(I-\mathcal{D}_\gamma)^{-1}$, then it yields that
    \begin{equation*}
      \begin{split}
        \|(I-\mathcal{D}_\gamma)^{-1}(I+\mathcal{D}_\gamma)\|_2
            &=\|(I+\mathcal{D}_\gamma)(I-\mathcal{D}_\gamma)^{-1}\|_2 \\
            &=\sup_{v\neq 0}\frac{v^\mathrm{T}(I-\mathcal{D}_\gamma^\mathrm{T})^{-1}(I+\mathcal{D}_\gamma^\mathrm{T})(I+\mathcal{D}_\gamma)(I-\mathcal{D}_\gamma)^{-1}v}{v^\mathrm{T}v} \\
            &\le 1.
      \end{split}
    \end{equation*}
\end{proof}
\begin{theorem}
  Let $u_{i,j}^{n}$ be the exact solution of \eqref{eq:4.1} with $1<\alpha,\beta\le2$, and $U_{i,j}^{n}$ the solution of the difference scheme \eqref{eq:4.7}, then for all $1\leq n\leq M$, we have
  \begin{equation}
    \|u^n-U^n\|\leq c(\tau^2+h^2),
  \end{equation}
  where c denotes the positive constant and $\|\cdot\|$ stands for the discrete $L^2$-norm.
\end{theorem}
\begin{proof}
  Let $e_{i,j}^{n}=u_{i,j}^{n}-U_{i,j}^{n}$, subtracting \eqref{eq:4.6} from
  \eqref{eq:4.7} leads to
  \begin{equation}
      \big(I-\mathcal{D}_x\big)\big(I-\mathcal{D}_y\big)e^{n+1}=\big(I+\mathcal{D}_x\big)
      \big(I+\mathcal{D}_y\big)e^{n}+\tau \mathcal{E}^{n},
  \end{equation}
  where $\mathcal{D}_x$ and $\mathcal{D}_y$ are given in \ref{eq:4.13} and
  \begin{align*}
    &e=(e_{1,1},e_{2,1},\cdots,e_{N-1,1},e_{1,2},e_{2,2},\cdots,e_{N-1,2},\cdots,e_{1,N-1},e_{2,N-1},\cdots,e_{N-1,N-1})^{\mathrm{T}},    \\
    &\mathcal{E}=(\varepsilon_{1,1},\varepsilon_{2,1},\cdots,\varepsilon_{N-1,1},\varepsilon_{1,2},\varepsilon_{2,2},\cdots,\varepsilon_{N-1,2},\cdots,\varepsilon_{1,N-1},\varepsilon_{2,N-1},\cdots,\varepsilon_{N-1,N-1})^{\mathrm{T}}.
  \end{align*}
  Since $\mathcal{D}_x$ commutes with $\mathcal{D}_y$, denoting $P=\big(I-\mathcal{D}_x\big)^{-1}\big(I-\mathcal{D}_y\big)^{-1}\big(I+\mathcal{D}_x\big)\big(I+\mathcal{D}_y\big)$, it yields that
  \begin{equation}
  e^{n+1}=Pe^{n}+\tau (I-\mathcal{D}_x\big)^{-1}\big(I-\mathcal{D}_y\big)^{-1}\mathcal{E}^n.
  \end{equation}
  Iterating for all $0\le k\le n-1$ and taking the $L^2$-norm on both sides, we have that
  \begin{equation}
  \|e^{n}\|\le\tau\|(I-\mathcal{D}_x)^{-1}(I-\mathcal{D}_y)^{-1}\|_2\sum_{k=0}^{n-1}\|P^{k}\|_2\cdot\|\mathcal{E}^{n-1-k}\|\le \tau\sum_{k=0}^{n-1}\|P^{k}\|_2\cdot\|\mathcal{E}^{n-1-k}\|,
  \end{equation}
  where Lemma \ref{lem:11} shows that $\|(I-\mathcal{D}_x)^{-1}(I-\mathcal{D}_y)^{-1}\|_2\le 1$.

  Since $\mathcal{D}_x$ and $\mathcal{D}_y$ commutes, matrix $P$ can be
  rewritten as
  \begin{equation}
  P=(I-\mathcal{D}_x)^{-1}(I+\mathcal{D}_x)(I-\mathcal{D}_y)^{-1}(I+\mathcal{D}_y).
  \end{equation}
  We then obtain from Lemma \ref{lem:11} that
  \begin{equation}
     \|P\|_2\le \|(I-\mathcal{D}_x)^{-1}(I+\mathcal{D}_x)\|_2\|(I-\mathcal{D}_y)^{-1}(I+\mathcal{D}_y)\|_2\le 1.
  \end{equation}
  Then for any $1\le k \le M$, $\|P^k\|_2\le \|P\|_2^k\le 1$ holds. We can get that
 \begin{equation}
  \|e^n\|\le \tau \sum_{k=0}^{n-1}\|\mathcal{E}^k\|\le c(\tau^2+h^2).
  \end{equation}
\end{proof}
%}
The convergence result for scheme \eqref{eq:4.10} can also be
obtained by the similar way as above.

%\mblue{

% }
%=========================================================================================
\section{Numerical Examples}
%-----------------------------------------------------------------------------------------
\subsection{One Dimensional Case}
\begin{example}\label{exm:1}
  Consider the following problem
  \begin{equation}\label{eq:5.1}
    \frac{\partial u(x,t)}{\partial t}={_0}D_x^{\alpha}u(x,t)-\mathrm{e}^{-t}\big(x^{1+\alpha}+\Gamma(2+\alpha)x\big),
    \quad (x,t)\in (0,1)\times(0,1],\\
  \end{equation}
  with the boundary conditions
  \begin{equation*}
    u(0,t)=0,\ \ u(1,t)=\mathrm{e}^{-t}, \quad t\in [0, 1],
  \end{equation*}
  and initial value
  \begin{equation*}
    u(x,0)=x^{1+\alpha}, \quad x\in [0, 1].
  \end{equation*}
  Then the exact solution of \eqref{eq:5.1} is $u(x,t)=\mathrm{e}^{-t}x^{1+\alpha}$.
\end{example}
\begin{table}[h]\fontsize{10pt}{12pt}\selectfont
  \begin{center}%\def\tabcolsep{28.5pt}
  \caption{The maximum and $L^2$ errors and their convergence rates to Example \ref{exm:1} approximated by the CN-WSGD scheme at $t=1$ for different $\alpha$ with $\tau=h$.}\vspace{5pt}
  \begin{tabular*}{\linewidth}{@{\extracolsep{\fill}}*{2}{r}*{8}{c}}
    \toprule
     & & \multicolumn{4}{c}{$(p,q)=(1,0)$} & \multicolumn{4}{c}{$(p,q)=(1,-1)$} \\
    \cline{3-6} \cline{7-10} \\[-8pt]
    $\alpha$ & $N$ & $\|u^n-U^n\|_{\infty}$ & rate & $\|u^n-U^n\|$ & rate
    & $\|u^n-U^n\|_{\infty}$ & rate & $\|u^n-U^n\|$ & rate\\
    \toprule
    1.1 & 16 & 6.65881E-05 & -  & 3.61993E-05 & - & 9.07705E-04 & -  & 9.88412E-05 & - \\
    & 32 & 1.54190E-05 & 2.11  & 8.91288E-06 & 2.02 & 2.28231E-04 & 1.99  & 1.69497E-05 & 2.54 \\
    & 64 & 3.59204E-06 & 2.10  & 2.20864E-06 & 2.01 & 5.54453E-05 & 2.04  & 3.18905E-06 & 2.41 \\
    & 128 & 8.38779E-07 & 2.10  & 5.50064E-07 & 2.01 & 1.32272E-05 & 2.07  & 6.62381E-07 & 2.27 \\
    & 256 & 2.07953E-07 & 2.01  & 1.37309E-07 & 2.00 & 3.12360E-06 & 2.08  & 1.49541E-07 & 2.15 \\
    & 512 & 5.19919E-08 & 2.00  & 3.43071E-08 & 2.00 & 7.33195E-07 & 2.09  & 3.55944E-08 & 2.07 \\
    \midrule
    1.5 & 16 & 6.17157E-05 & -  & 8.80121E-06 & - & 3.88221E-04 & -  & 3.91200E-05 & - \\
    & 32 & 1.25568E-05 & 2.30  & 2.30799E-06 & 1.93 & 7.85748E-05 & 2.30  & 5.04830E-06 & 2.95 \\
    & 64 & 2.47412E-06 & 2.34  & 6.07043E-07 & 1.93 & 1.54572E-05 & 2.35  & 7.43659E-07 & 2.76 \\
    & 128 & 4.76404E-07 & 2.38  & 1.56527E-07 & 1.96 & 2.97507E-06 & 2.38  & 1.49956E-07 & 2.31 \\
    & 256 & 9.01282E-08 & 2.40  & 3.97926E-08 & 1.98 & 5.62846E-07 & 2.40  & 3.72282E-08 & 2.01 \\
    & 512 & 1.93161E-08 & 2.22  & 1.00351E-08 & 1.99 & 1.05033E-07 & 2.42  & 9.60334E-09 & 1.95 \\
    \midrule
    1.9 & 16 & 1.63058E-05 & -  & 2.27814E-06 & - & 6.02603E-05 & -  & 7.78084E-06 & - \\
    & 32 & 2.49190E-06 & 2.71  & 6.49029E-07 & 1.81 & 9.23273E-06 & 2.71  & 9.04790E-07 & 3.10 \\
    & 64 & 4.93027E-07 & 2.34  & 1.81207E-07 & 1.84 & 1.35841E-06 & 2.76  & 1.49823E-07 & 2.59 \\
    & 128 & 1.27340E-07 & 1.95  & 4.81095E-08 & 1.91 & 1.94436E-07 & 2.80  & 4.02142E-08 & 1.90 \\
    & 256 & 3.23580E-08 & 1.98  & 1.24022E-08 & 1.96 & 3.06615E-08 & 2.66  & 1.12278E-08 & 1.84 \\
    & 512 & 8.15631E-09 & 1.99  & 3.14892E-09 & 1.98 & 7.93775E-09 & 1.95  & 2.99226E-09 & 1.91 \\
    \toprule
  \end{tabular*}\label{tab:1}
  \end{center}
\end{table}
\begin{example}\label{exm:2}
  Consider the following problem
  \begin{equation}\label{eq:5.2}
    \begin{split}
      & \frac{\partial u(x,t)}{\partial t}={_0}D_x^{\alpha}u(x,t)+{_x}D_1^{\alpha}u(x,t)
        +f(x,t), \quad (x,t)\in (0,1)\times(0,1],\\
      & u(0,t)=u(1,t)=0, \quad t\in [0, 1],\\
      & u(x,0)=x^3(1-x)^3, \quad x\in [0, 1],
    \end{split}
  \end{equation}
  with the source term
  \begin{equation*}
    \begin{split}
    f(x,t)=-\mathrm{e}^{-t}\Big(x^3(1-x)^3+\frac{\Gamma(4)}{\Gamma(4-\alpha)}\big(x^{3-\alpha}+(1-x)^{3-\alpha}\big)
                    -3\frac{\Gamma(5)}{\Gamma(5-\alpha)}\big(x^{4-\alpha}+(1-x)^{4-\alpha}\big)\\
                    +3\frac{\Gamma(6)}{\Gamma(6-\alpha)}\big(x^{5-\alpha}+(1-x)^{5-\alpha}\big)
                     -\frac{\Gamma(7)}{\Gamma(7-\alpha)}\big(x^{6-\alpha}+(1-x)^{6-\alpha}\big)\Big).
    \end{split}
  \end{equation*}
  By simple evaluation, the exact solution of \eqref{eq:5.2} is $u(x,t)=\mathrm{e}^{-t}x^3(1-x)^3$.
\end{example}
\begin{table}[htp]\fontsize{10pt}{12pt}\selectfont
  \begin{center}%\def\tabcolsep{28.5pt}
  \caption{The maximum and $L^2$ errors and their convergence rates to Example \ref{exm:2} approximated by the CN-WSGD scheme at $t=1$ for different $\alpha$ with $\tau=h$.}\vspace{5pt}
  \begin{tabular*}{\linewidth}{@{\extracolsep{\fill}}*{2}{r}*{8}{c}}
    \toprule
     & & \multicolumn{4}{c}{$(p,q)=(1,0)$} & \multicolumn{4}{c}{$(p,q)=(1,-1)$} \\
    \cline{3-6} \cline{7-10} \\[-8pt]
    $\alpha$ & $N$ & $\|u^n-U^n\|_{\infty}$ & rate & $\|u^n-U^n\|$ & rate
    & $\|u^n-U^n\|_{\infty}$ & rate & $\|u^n-U^n\|$ & rate\\
    \toprule
    1.1 & 16 & 1.21351E-04 & -  & 6.87244E-05 & - & 1.04202E-04 & -  & 5.49761E-05 & - \\
    & 32 & 3.10400E-05 & 1.97  & 1.75798E-05 & 1.97 & 4.32767E-05 & 1.27  & 2.00595E-05 & 1.45 \\
    & 64 & 7.93983E-06 & 1.97  & 4.47207E-06 & 1.97 & 1.48399E-05 & 1.54  & 7.42486E-06 & 1.43 \\
    & 128 & 2.01674E-06 & 1.98  & 1.12995E-06 & 1.98 & 4.19788E-06 & 1.82  & 2.23601E-06 & 1.73 \\
    & 256 & 5.08051E-07 & 1.99  & 2.84150E-07 & 1.99 & 1.10967E-06 & 1.92  & 6.11319E-07 & 1.87 \\
    & 512 & 1.27511E-07 & 1.99  & 7.12580E-08 & 2.00 & 2.84899E-07 & 1.96  & 1.59692E-07 & 1.94 \\
    \midrule
    1.5 & 16 & 2.03009E-04 & -  & 5.46438E-05 & - & 2.99388E-04 & -  & 8.57787E-05 & - \\
    & 32 & 4.52559E-05 & 2.17  & 1.37190E-05 & 1.99 & 7.90624E-05 & 1.92  & 2.31127E-05 & 1.89 \\
    & 64 & 1.13225E-05 & 2.00  & 3.45401E-06 & 1.99 & 2.01483E-05 & 1.97  & 6.01008E-06 & 1.94 \\
    & 128 & 2.83579E-06 & 2.00  & 8.67756E-07 & 1.99 & 5.08147E-06 & 1.99  & 1.53528E-06 & 1.97 \\
    & 256 & 7.09655E-07 & 2.00  & 2.17555E-07 & 2.00 & 1.27542E-06 & 1.99  & 3.88274E-07 & 1.98 \\
    & 512 & 1.77509E-07 & 2.00  & 5.44715E-08 & 2.00 & 3.19447E-07 & 2.00  & 9.76542E-08 & 1.99 \\
    \midrule
    1.9 & 16 & 2.02959E-04 & -  & 3.60448E-05 & - & 2.35899E-04 & -  & 4.37067E-05 & - \\
    & 32 & 4.57927E-05 & 2.15  & 8.97441E-06 & 2.01 & 5.44882E-05 & 2.11  & 1.10506E-05 & 1.98 \\
    & 64 & 9.36312E-06 & 2.29  & 2.23928E-06 & 2.00 & 1.13848E-05 & 2.26  & 2.77301E-06 & 1.99 \\
    & 128 & 2.03859E-06 & 2.20  & 5.59714E-07 & 2.00 & 2.55286E-06 & 2.16  & 6.94607E-07 & 2.00 \\
    & 256 & 5.08948E-07 & 2.00  & 1.39944E-07 & 2.00 & 6.35234E-07 & 2.01  & 1.73827E-07 & 2.00 \\
    & 512 & 1.27160E-07 & 2.00  & 3.49898E-08 & 2.00 & 1.58420E-07 & 2.00  & 4.34792E-08 & 2.00 \\
    \toprule
  \end{tabular*}\label{tab:2}
  \end{center}
\end{table}
\begin{example}\label{exm:3}
  Consider the following variable coefficients problem
  \begin{equation}\label{eq:5.3}
    \begin{split}
      & \frac{\partial u(x,t)}{\partial t}=x^\alpha{_0}D_x^{\alpha}u(x,t)+(1-x)^\alpha{_x}D_1^{\alpha}u(x,t)
        +f(x,t), \quad (x,t)\in (0,1)\times(0,1],\\
      & u(0,t)=u(1,t)=0, \quad t\in [0, 1],\\
      & u(x,0)=x^3(1-x)^3, \quad x\in [0, 1],
    \end{split}
  \end{equation}
  with the source term
  \begin{equation*}
    \begin{split}
    f(x,t)=-\mathrm{e}^{-t}\Big(x^3(1-x)^3+\frac{\Gamma(4)}{\Gamma(4-\alpha)}\big(x^{3}+(1-x)^{3}\big)
                    -3\frac{\Gamma(5)}{\Gamma(5-\alpha)}\big(x^{4}+(1-x)^{4}\big)\\
                    +3\frac{\Gamma(6)}{\Gamma(6-\alpha)}\big(x^{5}+(1-x)^{5}\big)
                     -\frac{\Gamma(7)}{\Gamma(7-\alpha)}\big(x^{6}+(1-x)^{6}\big)\Big).
    \end{split}
  \end{equation*}
  By simple evaluation, the exact solution of \eqref{eq:5.3} is $u(x,t)=\mathrm{e}^{-t}x^3(1-x)^3$.
\end{example}
\begin{table}[htp]\fontsize{10pt}{12pt}\selectfont
  \begin{center}%\def\tabcolsep{28.5pt}
  \caption{The maximum and $L^2$ errors and their convergence rates to Example \ref{exm:3} approximated by the CN-WSGD scheme at $t=1$ for different $\alpha$ with $\tau=h$.}\vspace{5pt}
  \begin{tabular*}{\linewidth}{@{\extracolsep{\fill}}*{2}{r}*{8}{c}}
    \toprule
     & & \multicolumn{4}{c}{$(p,q)=(1,0)$} & \multicolumn{4}{c}{$(p,q)=(1,-1)$} \\
    \cline{3-6} \cline{7-10} \\[-8pt]
    $\alpha$ & $N$ & $\|u^n-U^n\|_{\infty}$ & rate & $\|u^n-U^n\|$ & rate
    & $\|u^n-U^n\|_{\infty}$ & rate & $\|u^n-U^n\|$ & rate\\
    \toprule
    1.1 & 16 & 1.77123E-04 & -  & 7.32001E-05 & - & 3.95613E-04 & -  & 1.92219E-04 & - \\
    & 32 & 4.47870E-05 & 1.98  & 1.76184E-05 & 2.05 & 9.75763E-05 & 2.02  & 4.11452E-05 & 2.22 \\
    & 64 & 1.08962E-05 & 2.04  & 4.36356E-06 & 2.01 & 2.43654E-05 & 2.00  & 1.00363E-05 & 2.04 \\
    & 128 & 2.66784E-06 & 2.03  & 1.08906E-06 & 2.00 & 6.10991E-06 & 2.00  & 2.51523E-06 & 2.00 \\
    & 256 & 6.67126E-07 & 2.00  & 2.72235E-07 & 2.00 & 1.53026E-06 & 2.00  & 6.31764E-07 & 1.99 \\
    \midrule
    1.5 & 16 & 1.88510E-04 & -  & 6.18902E-05 & - & 3.56874E-04 & -  & 1.30433E-04 & - \\
    & 32 & 4.48741E-05 & 2.07  & 1.46628E-05 & 2.08 & 8.32954E-05 & 2.10  & 2.80619E-05 & 2.22 \\
    & 64 & 1.10524E-05 & 2.02  & 3.61334E-06 & 2.02 & 2.02076E-05 & 2.04  & 6.65178E-06 & 2.08 \\
    & 128 & 2.74933E-06 & 2.01  & 8.99424E-07 & 2.01 & 4.98975E-06 & 2.02  & 1.63398E-06 & 2.03 \\
    & 256 & 6.86120E-07 & 2.00  & 2.24518E-07 & 2.00 & 1.24092E-06 & 2.01  & 4.05976E-07 & 2.01 \\
    \midrule
    1.9 & 16 & 1.61881E-04 & -  & 4.02897E-05 & - & 1.79407E-04 & -  & 5.75044E-05 & - \\
    & 32 & 3.43080E-05 & 2.24  & 9.58213E-06 & 2.07 & 4.06728E-05 & 2.14  & 1.27751E-05 & 2.17 \\
    & 64 & 7.72475E-06 & 2.15  & 2.35289E-06 & 2.03 & 9.30708E-06 & 2.13  & 3.02268E-06 & 2.08 \\
    & 128 & 1.91676E-06 & 2.01  & 5.83977E-07 & 2.01 & 2.34573E-06 & 1.99  & 7.37420E-07 & 2.04 \\
    & 256 & 4.80573E-07 & 2.00  & 1.45527E-07 & 2.00 & 5.92611E-07 & 1.98  & 1.82315E-07 & 2.02 \\
    \toprule
  \end{tabular*}\label{tab:3}
  \end{center}
\end{table}
%-----------------------------------------------------------------------------------------
\subsection{Two Dimensional Case}
\begin{example}\label{exm:4}
  The following fractional diffusion problem
  \begin{equation*}
      \frac{\partial u(x,y,t)}{\partial t}={_0}D_x^{1.2}u(x,y,t)+{_x}D_1^{1.2}u(x,y,t)
        +{_0}D_y^{1.8}u(x,y,t)+{_y}D_1^{1.8}u(x,y,t)+f(x,y,t)
  \end{equation*}
  is considered in the domain $\Omega=(0,1)^2$ and $t>0$ with boundary conditions $u(x,y,t)|_{\partial\Omega}=0$ and the initial condition $u(x,y,0)=x^3(1-x)^3y^3(1-y)^3$, where the source term
  \begin{equation*}
    \begin{split}
    f(x,y,t)=-\mathrm{e}^{-t}\Big[&\Big(x^3(1-x)^3y^3(1-y)^3\Big)
              +\Big(\frac{\Gamma(4)}{\Gamma(2.8)}\big(x^{1.8}+(1-x)^{1.8}\big)
              -\frac{3\Gamma(5)}{\Gamma(3.8)}\big(x^{2.8}+(1-x)^{2.8}\big)\\
             &+\frac{3\Gamma(6)}{\Gamma(4.8)}\big(x^{3.8}+(1-x)^{3.8}\big)
              -\frac{\Gamma(7)}{\Gamma(5.8)}\big(x^{4.8}+(1-x)^{4.8}\big)\Big)y^3(1-y)^3\\
             &+\Big(\frac{\Gamma(4)}{\Gamma(2.2)}\big(y^{1.2}+(1-y)^{1.2}\big)
              -\frac{3\Gamma(5)}{\Gamma(3.2)}\big(y^{2.2}+(1-y)^{2.2}\big)\\
             &+\frac{3\Gamma(6)}{\Gamma(4.2)}\big(y^{3.2}+(1-y)^{3.2}\big)
              -\frac{\Gamma(7)}{\Gamma(5.2)}\big(y^{4.2}+(1-y)^{4.2}\big)\Big)x^3(1-x)^3\Big].
    \end{split}
  \end{equation*}
  Then the exact solution of the fractional partial differential equation is $u(x,y,t)=\mathrm{e}^{-t}x^3(1-x)^3y^3(1-y)^3$.
\end{example}
We use four numerical schemes: LOD \eqref{eq:4.15}, PR-ADI \eqref{eq:4.16}, Douglas-ADI
\eqref{eq:4.17} and D'yakonov-ADI \eqref{eq:4.18}, to simulate Example \ref{exm:4}, the
maximum and $L^2$ errors and their convergence rates to Example \ref{exm:4} approximated
at $t=1$ are listed in \autoref{tab:4}, where $N=N_x=N_y$, and $p,~q$ are the shifted
numbers of the WSGD operators. From the numerical results, three ADI schemes obtain more
accurate solution than the LOD scheme, and it also reflects that the three ADI schemes
are equivalent in two dimensional case.
\begin{table}[htp]\fontsize{10pt}{12pt}\selectfont
  \begin{center}%\def\tabcolsep{28.5pt}
  \caption{The maximum and $L^2$ errors and their convergence rates to Example \ref{exm:4} approximated at $t=1$ with $\tau=h$.}\vspace{5pt}
  \begin{tabular*}{\linewidth}{@{\extracolsep{\fill}}*{1}{l}*{1}{r}*{8}{c}}
    \toprule
     & & \multicolumn{4}{c}{$(p,q)=(1,0)$} & \multicolumn{4}{c}{$(p,q)=(1,-1)$} \\
    \cline{3-6} \cline{7-10} \\[-8pt]
    Scheme & $N$ & $\|u^n-U^n\|_{\infty}$ & ratio & $\|u^n-U^n\|$ & ratio & $\|u^n-U^n\|_{\infty}$ & ratio & $\|u^n-U^n\|$ & ratio\\
    \toprule
    & 8 & 4.49810E-05 & - & 1.36781E-05 & - & 4.81859E-05 & - & 1.50257E-05 & - \\
    & 16 & 1.16951E-05 & 1.94  & 3.68935E-06 & 1.89 & 1.21720E-05 & 1.99  & 3.77002E-06 & 1.99 \\
    LOD & 32 & 2.94559E-06 & 1.99  & 9.40245E-07 & 1.97 & 3.11386E-06 & 1.97  & 9.74178E-07 & 1.95 \\
    & 64 & 7.36186E-07 & 2.00  & 2.36472E-07 & 1.99 & 7.84850E-07 & 1.99  & 2.47973E-07 & 1.97 \\
    & 128 & 1.83637E-07 & 2.00  & 5.92494E-08 & 2.00 & 1.96486E-07 & 2.00  & 6.25130E-08 & 1.99 \\
    \midrule
    & 8 & 6.43195E-06 & - & 1.95007E-06 & - & 6.44770E-06 & - & 2.05016E-06 & - \\
    & 16 & 1.54712E-06 & 2.06  & 4.84833E-07 & 2.01 & 2.04790E-06 & 1.65  & 6.06100E-07 & 1.76 \\
    PR-ADI & 32 & 3.83522E-07 & 2.01  & 1.21460E-07 & 2.00 & 5.56723E-07 & 1.88  & 1.69028E-07 & 1.84 \\
    & 64 & 9.57751E-08 & 2.00  & 3.04854E-08 & 1.99 & 1.44070E-07 & 1.95  & 4.50482E-08 & 1.91 \\
    & 128 & 2.39462E-08 & 2.00  & 7.64237E-09 & 2.00 & 3.65748E-08 & 1.98  & 1.16567E-08 & 1.95 \\
    \midrule
    & 8 & 6.43195E-06 & - & 1.95007E-06 & - & 6.44770E-06 & - & 2.05016E-06 & - \\
    Douglas-& 16 & 1.54712E-06 & 2.06  & 4.84833E-07 & 2.01 & 2.04790E-06 & 1.65  & 6.06100E-07 & 1.76 \\
    ADI & 32 & 3.83522E-07 & 2.01  & 1.21460E-07 & 2.00 & 5.56723E-07 & 1.88  & 1.69028E-07 & 1.84 \\
    & 64 & 9.57751E-08 & 2.00  & 3.04854E-08 & 1.99 & 1.44070E-07 & 1.95  & 4.50482E-08 & 1.91 \\
    & 128 & 2.39462E-08 & 2.00  & 7.64237E-09 & 2.00 & 3.65748E-08 & 1.98  & 1.16567E-08 & 1.95 \\
    \midrule
    & 8 & 6.43195E-06 & - & 1.95007E-06 & - & 6.44770E-06 & - & 2.05016E-06 & - \\
    D'yakonov-& 16 & 1.54712E-06 & 2.06  & 4.84833E-07 & 2.01 & 2.04790E-06 & 1.65  & 6.06100E-07 & 1.76 \\
    ADI & 32 & 3.83522E-07 & 2.01  & 1.21460E-07 & 2.00 & 5.56723E-07 & 1.88  & 1.69028E-07 & 1.84 \\
    & 64 & 9.57751E-08 & 2.00  & 3.04854E-08 & 1.99 & 1.44070E-07 & 1.95  & 4.50482E-08 & 1.91 \\
    & 128 & 2.39462E-08 & 2.00  & 7.64237E-09 & 2.00 & 3.65748E-08 & 1.98  & 1.16567E-08 & 1.95 \\
    \toprule
  \end{tabular*}\label{tab:4}
  \end{center}
\end{table}
%=========================================================================================
\section{Conclusion}
The paper provides the novel second order approximations for
fractional derivatives, called the weighted and shifted Gr\"{u}nwald
difference operator; it also suggests a direction to gain higher
order discretization and compact schemes of fractional derivatives.
The discretizations are used to solve one and two dimensional space
fractional diffusion equations; several numerical schemes are
designed, their effectiveness are theoretically proved and
numerically verified.

\section{Acknowledgements}
The authors thank Prof Yujiang Wu for his constant encouragement and
support. This work was supported by the Program for New Century
Excellent Talents in University under Grant No. NCET-09-0438, the
National Natural Science Foundation of China under Grant No.
10801067, and the Fundamental Research Funds for the Central
Universities under Grant No. lzujbky-2010-63 and No.
lzujbky-2012-k26.

%=========================================================================================
%------------------------------------ Bibliography ---------------------------------------

%=========================================================================================
\end{document}